\documentclass[11pt]{amsart}
\usepackage{txfonts}
\usepackage{amsmath,amsthm,verbatim}
\usepackage{amscd,amssymb}
\usepackage[all]{xy} 
\usepackage{graphicx,calrsfs} 
\usepackage[colorlinks,plainpages,backref,urlcolor=blue,breaklinks=true]{hyperref}
\usepackage{breakurl}
\usepackage{bbm}
\usepackage{mathrsfs}
\usepackage{xstring}
\usepackage{tikz}
\usetikzlibrary{matrix,arrows,decorations.pathmorphing,calc,shapes}
\usepackage{enumitem}

\newtheorem{theorem}{Theorem}[section]
\newtheorem{corollary}[theorem]{Corollary}
\newtheorem{lemma}[theorem]{Lemma}
\newtheorem{prop}[theorem]{Proposition}

\theoremstyle{definition}
\newtheorem{definition}[theorem]{Definition}
\newtheorem{example}[theorem]{Example}
\newtheorem{remark}[theorem]{Remark}

\DeclareMathOperator{\im}{im}
\DeclareMathOperator{\ini}{in}
\DeclareMathOperator{\id}{id}

\renewcommand{\hom}{\operatorname{Hom}}

\DeclareMathOperator{\rank}{rank}

\DeclareMathOperator{\ab}{ab}

\DeclareMathOperator{\gr}{gr}
\DeclareMathOperator{\Sym}{Sym}
\DeclareMathOperator{\Hilb}{Hilb}

\DeclareMathOperator{\GL}{GL}

\DeclareMathOperator{\ideal}{ideal}

\DeclareMathOperator{\lk}{lk}

\newcommand{\R}{\mathbb{R}}
\newcommand{\Q}{\mathbb{Q}}

\newcommand{\Z}{\mathbb{Z}}

\renewcommand{\Q}{\mathbb{Q}}
\newcommand{\g}{\mathfrak{g}}
\newcommand{\h}{\mathfrak{h}}

\newcommand{\bx}{\mathbf{x}}
\newcommand{\br}{\mathbf{r}}
\newcommand{\by}{\mathbf{y}}
\newcommand{\bw}{\mathbf{w}}

\newcommand{\fh}{\mathfrak{h}}
\newcommand{\fg}{\mathfrak{g}}

\newcommand{\fm}{\mathfrak{m}}

\newcommand{\Lie}{\mathfrak{lie}}

\newcommand{\cF}{\mathcal{F}}

\DeclareMathAlphabet{\pazocal}{OMS}{zplm}{m}{n}

\newcommand{\surj}{\twoheadrightarrow}

\newcommand{\abs}[1]{\left| #1 \right|}
\newcommand{\ncom}[1]{\langle\!\langle #1 \rangle\!\rangle}
\newcommand\isom{\xrightarrow{
 \,\smash{\raisebox{-0.6ex}{\ensuremath{\scriptstyle\simeq}}}\,}}

\def\dot{\mathchar"013A}  
\newcommand{\hdot}{{\raise1pt\hbox to0.35em{\huge $\dot$}}}

\begin{document}


\title[Cup products, lower central series, and holonomy Lie algebras]%
{Cup products, lower central series, and \\ holonomy Lie algebras}

\author[Alexander~I.~Suciu]{Alexander~I.~Suciu$^1$}
\address{Department of Mathematics,
Northeastern University,
Boston, MA 02115, USA}
\email{\href{mailto:a.suciu@northeastern.edu}{a.suciu@northeastern.edu}}
\urladdr{\href{http://web.northeastern.edu/suciu/}%
{http://web.northeastern.edu/suciu/}}
\thanks{$^1$Supported in part by the National Security Agency 
(grant H98230-13-1-0225) and the Simons Foundation 
(collaboration grant for mathematicians 354156)}

\author{He Wang}
\address{Department of Mathematics and Statistics,
University of Nevada, Reno, NV 89557, USA}
\email{\href{mailto:wanghemath@gmail.com}%
{hew@unr.edu}}
\urladdr{\href{http://wolfweb.unr.edu/homepage/hew/}%
{http://wolfweb.unr.edu/homepage/hew/}}



\begin{abstract}
We generalize basic results relating the associated graded Lie algebra and 
the holonomy Lie algebra of a group, from finitely presented, commutator-relators 
groups to arbitrary finitely presented groups. Using the notion of ``echelon 
presentation," we give an explicit formula for the cup-product in the cohomology 
of a finite $2$-complex.  This yields an algorithm for computing the corresponding 
holonomy Lie algebra, based on a Magnus expansion method.  As an application, 
we discuss issues of graded-formality, filtered-formality, $1$-formality, and mildness. 
We illustrate our approach with examples drawn from a variety of group-theoretic 
and topological contexts, such as link groups, one-relator groups, and fundamental 
groups of orientable Seifert fibered manifolds.  
\end{abstract}

\maketitle
\setcounter{tocdepth}{1}
\tableofcontents

\section{Introduction}
\label{sect:intro}

Throughout this paper $G$ will be a finitely generated group. 
Our main focus will be on the cup-product in the rational 
cohomology of the $2$-complex associated to a presentation of $G$, and 
on several rational Lie algebras attached to such a group. 

\subsection{Magnus expansions and cup products} 
\label{intro:pres}
The notion of expansion of a group, which goes back to W.~Magnus 
\cite{Magnus35}, has been generalized and used in many ways.
For instance, a presentation for the Malcev Lie algebra of a finitely presented 
group was given by S.~Papadima \cite{Papadima95} and G.~Massuyeau \cite{Massuyeau12},   
while X.-S. Lin  \cite{Linxiaosong97}  studied expansions of fundamental groups 
of smooth manifolds.  Recently, D.~Bar-Natan \cite{Bar-Natan16} has generalized 
the notion of expansion and has introduced the Taylor expansion of an arbitrary ring. 
In turn, we explored  in \cite{SW-formality} various relationships between
expansions and formality properties of groups. 

We go back here to the classical Magnus expansion, and adapt it for 
our purposes.  Let $G$ be a group with a finite presentation 
$G=F/R=\langle x_1,\dots, x_n\mid r_1,\dots,r_m\rangle$. 
There exists then a $2$-complex $K=K_G$ associated to 
such a presentation.  
In the case when $G$ is a commutator-relators group, i.e., 
when all relators $r_i$ belong to the commutator subgroup $[F,F]$, 
R.~Fenn and D.~Sjerve  computed in \cite{Fenn-S} the cup-product map
\begin{equation}
\label{eq:cupMassey}
\xymatrixcolsep{20pt}
\xymatrix{\mu_K\colon H^1(K;\Z) \wedge H^1(K;\Z)\ar[r]& H^2(K;\Z)}, 
\quad u\wedge v \mapsto u\cup v,
\end{equation}
using  the Magnus expansion $M\colon \Z{F}\to \Z \ncom{\mathbf{x}}$
from the group ring of the free group $F=\langle x_1,\dots,x_n\rangle$ to 
the power series ring in $n$ non-commuting variables, which is the 
ring morphism defined by $M(x_i)=1+x_i$.

Our first objective in this work is to generalize this result of Fenn and Sjerve, 
from commutator-relators groups to arbitrary finitely presented groups. 
We will avoid possible torsion in the first homology of $G$ by working 
over the field of rationals.  To that end, we start by defining a Magnus-like 
expansion $\kappa=\kappa_G$ relative to such a group $G$ as the composition
\begin{equation}
\label{eq:kappa-intro}
\xymatrixrowsep{28pt}
\xymatrix{
\Q{F} \ar[r]^-M&  \widehat{T}(H_1(F;\Q))   
\ar^{\widehat{T}({\varphi_*})}[r] & \widehat{T}(H_1(G;\Q)),
}
\end{equation}  
where $\varphi\colon F\surj G$ is the canonical projection 
and $\widehat{T}(V)$ is the completed tensor algebra of a 
vector space $V$.
We then show in Proposition \ref{prop: basis} that there 
exists a group $G_e$ admitting  a `row-echelon' presentation,  
$G_e=\langle x_1,\dots, x_n\mid w_1,\dots,w_m\rangle$,
and a map $f\colon K_{G_e}\to K_{G}$ inducing an isomorphism 
in cohomology.

Using the $\kappa$-expansion of $G_e$, we determine in Theorem \ref{thm:cup G tilde} 
the cup-product map for $K_{G_e}$, from which we obtain in Theorem \ref{Thmcupproduct again} 
a formula for computing the cup-product map $\mu_K$, with $\Q$-coef\-ficients.  
Let $b=b_1(G)$ be the first Betti number of $G$, and let $\{u_{1},\dots, u_b\}$ 
and $\{\beta_{n-b+1},\dots, \beta_m\}$ be bases for $H^1(K;\Q)$ and $H^2(K;\Q)$, 
transferred from suitable bases in the rational cohomology of $K_{G_e}$ via the 
isomorphism $f^*\colon H^*(K_{G_e};\Q)\to H^*(K_{G};\Q)$. 
Our result then reads as follows.

\begin{theorem}
\label{thm:intro1}
Let $K$ be a presentation $2$-complex for a finitely presented group $G$. 
In the bases described above, the cup-product map $\mu_K\colon H^1(K;\Q) 
\wedge H^1(K;\Q)\rightarrow H^2(K;\Q)$  is  given by
\[
u_i \cup u_j=\sum_{k=n-b+1}^{m}  \kappa(w_k)_{i,j}\, \beta_k, \quad 
\textrm{for $1\leq i, j\leq b$}. 
\]
\end{theorem}

Let us note that the map $\mu_K$ depends on the chosen presentation for $G$, 
and may differ from the cup-product map $\mu_G$ in a classifying space for $G$. 
However, the two maps share the same kernel, which makes the algorithm for 
determining the map $\mu_K$ useful in other contexts,  
for instance, in computing the first resonance variety of $G$ 
(see e.g.~\cite{MS-tokyo, PS-crelle}), or  finding a presentation 
for the holonomy Lie algebra $\fh(G)$, a procedure that we discuss next.

\subsection{Holonomy Lie algebras} 
\label{intro:holo}

The {\em holonomy Lie algebra}\/ of a finitely generated group $G$, 
denoted by $\fh(G)$, is the quotient of the free Lie algebra on $H_1(G;\Q)$
by the Lie ideal generated by the image of the dual of the cup-product map 
$\mu_G$. It is easy to see that $\fh(G)$ is a graded Lie algebra over $\Q$ which 
admits a quadratic presentation depending only on $\ker \mu_G$.  
Moreover, this construction is functorial.  The holonomy Lie 
algebra was introduced by T.~Kohno in \cite{Kohno}, building on work of 
Chen \cite{Chen73}, and has been further studied in a number of papers, 
including \cite{Markl-Papadima, PS-imrn, SW-formality}.   

Our next objective is to find a presentation for the holonomy Lie algebra $\fh(G)$. 
We start by showing in Proposition \ref{prop:holo fp}
that there is a homomorphism from a finitely presented group $G_f$ 
to $G$ inducing an isomorphism on holonomy Lie algebras.   
Hence, without loss of generality, we may assume that $G$ admits
a finite presentation.  

Let $G_e=\langle x_1,\dots, x_n\mid w_1,\dots,w_m\rangle$ 
be a group with row-echelon presentation, and let 
$\rho\colon G_e\to G$ be the homomorphism 
induced on fundamental groups by the aforementioned 
map, $f\colon K_{G_e}\to K_G$.  
It follows from Corollary \ref{cor:holo map} that 
the induced map, $\h(\rho)\colon \h(G_e) \to \h(G)$, 
is an isomorphism of graded Lie algebras.   Using the computation of the 
cup-product map $\mu_{K_G}$ from Theorem \ref{thm:intro1}, 
we describe in Theorem \ref{ThmholonomyLie} an algorithm for finding a 
presentation for the holonomy Lie algebra $\fh(G)$.   
Furthermore, we obtain in Theorem \ref{thm:holo chen lie} 
a presentation for the derived quotients of this Lie algebra, 
$\fh(G)/\fh(G)^{(i)}$.  Our results may be summarized as follows.

\begin{theorem}
\label{thm:intro}
Let $G$ be a finitely presented group. 
The holonomy Lie algebra $\fh(G)$ is the quotient of the free 
$\Q$-Lie algebra with generators $\by=\{y_1,\dots, y_b\}$ in degree $1$ 
by the ideal $I$ generated by $\kappa_2(w_{n-b+1}),\dots,\kappa_2(w_{m})$, 
where $\kappa_2$ is the degree $2$ part of the Magnus expansion of $G_e$.
Furthermore, for each $i\ge 2$, the solvable quotient $\fh(G)/\fh(G)^{(i)}$ 
is isomorphic to $\Lie(\by)/( I + \Lie^{(i)}(\by))$. 
\end{theorem}

In the special case when $G$ admits a presentation with only 
commutator relators, presentations for these Lie algebras 
were given by Papadima and Suciu in \cite{PS-imrn}. 
For arbitrary finitely generated groups $G$, 
the metabelian quotient $\fh(G)/\fh(G)''$, also known as the 
holonomy Chen Lie algebra of $G$, is closely related to the first 
resonance variety of $G$, a geometric object which has been 
studied intensely from many points of view, see for instance 
\cite{MS-tokyo, PS-mrl, PS-crelle, SW-mz, SW-mccool} 
and references therein. 

\subsection{Lower central series, graded formality, and mildness}
\label{intro:gf mild} 

The Lie methods in group theory were introduced by 
W.~Magnus in \cite{Magnus37}, and further developed 
by E.~Witt \cite{Witt37}, M.~Hall \cite{Hall50}, M.~Lazard \cite{Lazard54}, 
and many more authors, see for instance \cite{Magnus-K-S}.  
The {\em associated graded Lie ring}\/ of a group $G$ 
is the graded Lie ring $\gr(G;\Z)$, whose graded 
pieces are the successive quotients of the lower central series of $G$, 
and whose Lie bracket is induced from the group commutator.   
The quintessential example is the associated graded Lie ring  
of the free group on $n$ generators, $F_n$, which is isomorphic 
to the free Lie ring  $\Lie(\Z^n)$. Much of the power of this method 
comes from the various connections between lower central series, 
nilpotent quotients, and group homology, as evidenced in the 
work of  J.~Stallings \cite{Stallings},  W.~Dwyer \cite{Dwyer}, and 
many others.  

We concentrate here on the associated graded Lie algebra over the  
rationals, $\gr(G)=\gr(G;\Z)\otimes \Q$, of a finitely generated group $G$.  
As we recall in \S\ref{subsec:Phimap}, there is a natural epimorphism 
$\Phi_G\colon \fh(G) \surj \gr(G)$.   Thus, the holonomy Lie algebra $\fh(G)$ 
may be viewed as a quadratic approximation to the associated graded Lie algebra of $G$.   
We say that the group $G$ is {\em graded-formal}\/ if the map $\Phi_G$ is an isomorphism 
of graded Lie algebras.  A much stronger requirement is that $G$ be {\em $1$-formal}, 
a condition we recall in \S\ref{subsec:completions}. For much more on these 
notions, we refer to \cite{PS-imrn, PS-bmssmr, SW-formality}. 

In Propositions \ref{prop:holo nq} and \ref{prop:derivedquo}, 
we compare the holonomy Lie algebra of $G$ with the holonomy 
Lie algebras of the nilpotent quotients $G/\Gamma_iG$ and the derived 
quotients $G/G^{(i)}$.  In Corollary  \ref{cor:chenlie-formal}, we use 
Theorem \ref{thm:intro}  and a result from \cite{SW-formality} to give 
explicit presentations for the graded Lie algebras $\gr\big(G/G^{(i)}\big)$ 
in the case when $G$ is a finitely presented, $1$-formal group. 

Some of the motivation for our study comes from the work of 
J.~Labute  \cite{Labute70, Labute85} and D.~Anick \cite{Anick87}, 
who gave presentations for the associated graded Lie algebra 
$\gr(G)$ in the case when $G$ is `mildly' presented.   
We revisit this topic in \S\ref{sect:mildness}, where we 
relate the notion of mild presentation to that of graded formality, 
and derive some consequences, especially in the context of link groups. 

\subsection{Further applications}
\label{intro:applications} 

We illustrate our approach with several classes of finitely presented groups,
including $1$-relator groups in \S\ref{sec:onerel}, and fundamental groups 
of orientable Seifert fibered $3$-manifolds with orientable base in \S\ref{sec:seifert mfd}.
We give here presentations for the holonomy Lie algebra $\h(G)$ and 
the Chen Lie algebra $\gr(G/G'')$ of such groups $G$.  
We also compute the Hilbert series of these graded Lie algebras, and 
discuss the formality properties of these groups.

This work was motivated by a desire to generalize some of the results 
of Fenn--Sjerve \cite{Fenn-S} and Papadima--Suciu \cite{PS-imrn}, 
from commutator-relators groups to 
arbitrary finitely generated groups.  In \cite{SW-formality}, we studied the 
formality properties of finitely generated groups,  focusing on the
filtered-formality and $1$-formality properties.  In related work, we 
apply the techniques developed in this paper and in \cite{SW-formality} 
to the study of several families of ``pure-braid like" groups.  For instance, 
we investigate in \cite{SW-mz} the pure virtual braid groups, and we investigate 
in \cite{SW-mccool} the McCool groups, also known as the pure welded 
braid groups.  A summary of these results, as well as further motivation 
and background can be found in \cite{SW-indam}. 

\section{Expansions for finitely presented groups}
\label{sect:q-magnus}

In this section, we introduce and study a Magnus-type expansion 
relative to a finitely presented group.  We start by reviewing some 
basic notions.

\subsection{Completed group algebras and expansions}
\label{subsec:completions}

Let $G$ be a finitely generated group.  As is well-known (see 
for instance \cite{Passi, Quillen68}), 
the group-algebra $\Q{G}$ has a natural Hopf algebra structure,
with comultiplication $\Delta\colon \Q{G}\to \Q{G}\otimes_{\Q}\Q{G}$ 
given by $\Delta(g)=g\otimes g$ for $g\in G$, and counit the 
augmentation map $\varepsilon\colon \Q G\rightarrow \Q$ 
given by $\varepsilon(g)=1$.  The powers of the augmentation 
ideal, $I=\ker \varepsilon$, form a descending, multiplicative filtration 
of $\Q{G}$.  The associated graded algebra, 
${\gr}(\Q{G})=\bigoplus_{k\geq 0}I^k/I^{k+1}$,  
comes endowed with the degree filtration, 
$\cF_k=\bigoplus_{j\geq k}I^j/I^{j+1}$.  
The corresponding completion, $\widehat{\gr}(\Q{G})$, 
is again an algebra, endowed with the inverse limit filtration. 

The $I$-adic completion of the group-algebra,  
$\widehat{\Q{G}}=\varprojlim_k \Q{G}/I^k$, 
also comes equipped with an inverse limit filtration. 
Applying the $I$-adic completion functor to the map 
$\Delta$ yields a comultiplication map   
$\widehat{\Delta}$, which makes $\widehat{\Q{G}}$ 
into a complete Hopf algebra, see~\cite[App.~A]{Quillen69}.

An element $x$ in a Hopf algebra is called {\em primitive}\/ if 
$\Delta x=x {\otimes} 1+ 1{\otimes} x$.
The set $\fm(G)$ of all primitive elements in $\widehat{\Q{G}}$, with bracket 
$[x,y]=xy-yx$, is a complete, filtered Lie algebra, called the \emph{Malcev Lie algebra}\/ of $G$.  
The set of all primitive elements in $\gr(\Q G)$ forms a graded Lie algebra,
which is isomorphic to the associated graded Lie algebra 
\begin{equation}
\label{eq:grg alg}
\gr(G):=\bigoplus\limits_{k\geq 1} \left(\Gamma_kG/\Gamma_{k+1}G\right) \otimes \Q,
\end{equation}
where $\{\Gamma_k G\}_{k\geq 1}$ is the \emph{lower central series}\/ of $G$,  
defined inductively by $\Gamma_1G=G$ and 
$\Gamma_{k+1} G=[\Gamma_kG,G]$ for $k\geq 1$. 
As shown by Quillen in \cite{Quillen68}, there is an isomorphism of graded 
Lie algebras, $\gr(\fm(G))\cong \gr(G)$.
 
The group $G$ is said to be {\em filtered-formal}\/ if its Malcev Lie algebra is isomorphic 
(as a filtered Lie algebra) to the degree completion of its associated graded Lie algebra.
The group $G$ is said to be {\em $1$-formal}\/ if its Malcev Lie algebra admits a 
quadratic presentation (see \cite{PS-bmssmr, SW-formality} for more details). 
For instance, all finitely generated free groups and free abelian groups 
are $1$-formal. 

An \emph{expansion}\/ for a group $G$ is a filtration-preserving 
algebra  morphism ${E}\colon\Q {G}\to \widehat{\gr}(\Q{G})$ with the 
property that $\gr({E})=\id$ 
(see \cite{Linxiaosong97, Bar-Natan16, SW-formality}). 
As shown in \cite{SW-formality}, a finitely generated group $G$ is 
filtered-formal if and only if it has an expansion $E$ which 
induces an isomorphism of complete Hopf algebras, 
$\widehat{E}\colon \widehat{\Q{G}}\to \widehat{\gr}(\Q{G})$.
 
\subsection{The Magnus expansion for a free group}
\label{subsec:magnus}
Let $F$ be a finitely generated free group, with generating set 
$\mathbf{x}=\{x_1,\dots,x_n\}$, and let $\Z{F}$ be its group-ring.  
Then the degree completion of the associated graded ring,  
$\widehat{\gr}(\Z{F})$, can be identified with the completed 
tensor ring $\widehat{T}(F_{\ab})=\Z\ncom{\bx}$, 
the power series ring over $\Z$ in $n$ non-commuting 
variables, by sending $[x_i-1]$ to $x_i$. 
There is a well known expansion 
$M\colon \Z{F}\to \Z \ncom{\bx}$, 
called the {\it Magnus expansion}, given by
\begin{equation}
\label{eq:agnus}
M(x_i)=1+x_i\quad \text{and} \quad  M(x_i^{-1})=1-x_i+x_i^2-x_i^3+\cdots.
\end{equation}
 
The \textit{Fox derivatives}\/ are the ring morphisms 
$\partial_i\colon \Z{F}\rightarrow  \Z{F}$ defined by the rules 
$\partial_i(1)=0$, $\partial_i(x_j)=\delta_{ij}$, and
$\partial_i(uv)=\partial_i(u)\varepsilon(v)+u\partial_i(v)$ for $u,v\in \Z{F}$,
where $\varepsilon\colon \Z{F}\to \Z$ is the augmentation map.
The higher Fox derivatives $\partial_{i_1,\dots, i_k}$ are then 
defined inductively. We refer to \cite{Fe,Fenn-S,Magnus-K-S,MS-tokyo} 
for more details and references on these notions.

The Magnus expansion can be computed in terms of the Fox derivatives, as follows.  
Given an element $y\in F$, if we write $M(y)=1+\sum a_I x_{I}$, then $a_I=\varepsilon_I(y)$,
where $I=(i_1,\dots,i_s)$, and $\varepsilon_I=\varepsilon \circ \partial_I$ is the composition of 
the augmentation map with the iterated Fox derivative $\partial_I\colon \Z{F}\to \Z{F}$.  
For each $k\ge 1$, let $M_k$ be the composite 
\begin{equation}
\label{eq:magnus}
\xymatrixcolsep{40pt}
\xymatrix{
\Z{F} \ar^(.46){M}[r] \ar^{M_k}@/^20pt/[rr] &  
\widehat{T}(F_{\ab})   \ar^(.46){\gr_k}[r] & 
 \gr_k(\widehat{T}(F_{\ab}))
}.
\end{equation}

For each $y$ in $F$, we have that $M_1(y)=\sum_{i=1}^n \varepsilon_i(y) x_i$, 
while for each $y$ in the commutator subgroup $[F,F]$, we have  
\begin{equation}
\label{eq:M2}
M_2(y)=\sum_{i<j} \varepsilon_{i,j}(y)(x_ix_j-x_jx_i).
\end{equation}   
 
The tensor algebra $T(F_{\Q})$ on the $\Q$-vector space 
$F_{\Q}=F_{\ab}\otimes \Q$ has a natural graded Hopf algebra structure, 
with comultiplication $\Delta$ and counit $\varepsilon$
given by $\Delta(a)= a\otimes 1+1\otimes a$ and $\varepsilon(a)= 0$ 
for $a\in F_{\Q}$.  The set of primitive elements in $T(F_{\Q})$ is the free
Lie algebra $\Lie(F_{\Q})=\{v \in T(F_{\Q}) \mid \Delta(v)=v \otimes 1+1\otimes v\}$,
with Lie bracket $[v,w]=v\otimes w-w\otimes v$.
Notice that, if $y\in[F,F]$, then $M_2(y)$ is a primitive element in the degree $2$ 
piece of the Hopf algebra ${T}(F_{\Q})= \gr(\widehat{T}(F_{\Q}))$,  
which corresponds to the degree $2$ 
element $\sum_{i<j} \varepsilon_{i,j}(y)[ x_i,x_j ]$ in the 
free Lie algebra $\Lie(F_{\Q})$.

\subsection{The Magnus expansion relative to a finitely generated group}
\label{subsec:kappa}

Given a finitely generated group $G$, there exists an epimorphism 
$\varphi\colon F\surj G$ from a free group $F$ of finite rank to $G$. 
Let $\varphi_{\ab}\colon F_{\ab}\surj G_{\ab}$ be the induced epimorphism 
between the respective abelianizations.  

\begin{definition}
\label{def:kappa}
The \textit{Magnus $\kappa$-expansion}\/ for $F$ relative to $G$, 
denoted by $\kappa_G$ (or $\kappa$ for short), is the composition 
\begin{equation}
\label{eq:magnus new}
\xymatrixcolsep{40pt}
\begin{gathered}
\xymatrix{
\Z{F} \ar^(.45){M}[r] \ar^{\kappa}@/^20pt/[rr] &  
\widehat{T}(F_{\ab})   \ar^{\widehat{T}({\varphi_{\ab}})}[r] & 
\widehat{T}(G_{\ab})
},
\end{gathered}
\end{equation}
where $M$ is the classical Magnus expansion for the free group $F$, 
and the morphism 
$\widehat{T}(\varphi_{\ab})\colon \widehat{T}(F_{\ab})\surj \widehat{T}(G_{\ab})$ 
is induced by the abelianization map $\varphi_{\ab}$.
\end{definition}

In particular, if the group $G$ is a commutator-relators group,
i.e., if all the relators of $G$ lie in the commutator subgroup $[F,F]$,
then the projection $\varphi_{\ab}$ identifies 
$F_{\ab}$ with $G_{\ab}$,  and the Magnus expansion 
$\kappa$ coincides with the classical Magnus expansion $M$.

More generally, let $G$ be a group generated by 
$\mathbf{x}=\{x_1,\dots,x_n\}$, and let $F$ be the free 
group generated by the same set.  
The rational Magnus $\kappa$-expansion, still denoted 
by $\kappa_G$ or $\kappa$, is the composition
\begin{equation}
\label{eq:magnus rational}
\xymatrixcolsep{40pt}
\xymatrix{
\Q{F} \ar^(.45){M}[r]   &  
\widehat{T}(F_{\Q})   \ar^{\widehat{T}({\pi})}[r] & 
\widehat{T}(G_{\Q})
},
\end{equation}
where $\pi=\varphi_{\ab}\otimes \Q=H_1(\varphi,\Q)$ is the induced epimorphism in homology from 
$F_{\Q}:=H_1(F;\Q)$ to $G_{\Q}:=H_1(G;\Q)$.
Pick a basis $\mathbf{y}=\{y_1,\dots,y_b\}$ for $G_{\Q}$, 
and identify $\widehat{T}(G_{\Q})$ with $\Q\ncom{\mathbf{y}}$.
Let $\kappa(r)_{I}$ be the coefficient of $y_{I}:=y_{i_1}\cdots y_{i_s}$ in $\kappa(r)$, 
for $I=(i_1,\dots, i_s)$. Then we can write 
\begin{equation}
\label{eq:kappar}
\kappa(r)=1+\sum_{I} \kappa(r)_{I}  y_{I} \, .
\end{equation}

\begin{lemma}
\label{lem:symmetry}
If $r\in \Gamma_k F$, 
then $\kappa(r)_I=0$, for $\abs{I}< k$.  Furthermore, 
if $r\in \Gamma_2 F$, then $\kappa(r)_{i,j}=-\kappa(r)_{j,i}$.
\end{lemma}

\begin{proof}
Since $M(r)_I=\varepsilon_I(r)=0$ for $\abs{I}< k$ (see for instance 
\cite{MS-tokyo}), we have that $\kappa(r)_I=0$ for $\abs{I}< k$.  
To prove the second assertion, identify the completed 
symmetric algebras $\widehat{\Sym}(F_{\Q})$ 
and $\widehat{\Sym}(G_{\Q})$ with the power series rings 
$\Q[\![\mathbf{x}]\!]$ and $\Q[\![\mathbf{y}]\!]$, respectively,  
in the following commuting diagram of $\Q$-linear maps.
\begin{equation}
\label{formula:sym}
\xymatrix{   
  \Q F\ar[dr]^{\kappa}\ar[r]^{ M} & \widehat{T}(F_{\Q})\ar[d]^{\widehat{T}(\pi)} 
  \ar[r]^{\alpha_1} & \widehat{\Sym}(F_{\Q})\ar[d]^{\widehat{\Sym}(\pi)} \phantom{\,.} \\
   & \widehat{T}(G_{\Q}) \ar[r]^{\alpha_2} &\widehat{\Sym}(G_{\Q}) \,.
}
\end{equation}

When $r\in[F,F]$, we have that $\alpha_2\circ \kappa(r)=
\widehat{\Sym}(\pi)\circ \alpha_1\circ M(r)=1$.
Thus, $\kappa_i(r)=0$ and $\kappa(r)_{i,j}+\kappa(r)_{j,i}=0$.
\end{proof}

\begin{lemma}
\label{lemmasum}
If $u,v \in F$ satisfy $\kappa(u)_J=\kappa(v)_J=0$ for all $\abs{J}<s$, 
for some $s\geq 2$, then 
\[
\kappa(uv)_{I}=\kappa(u)_{I}+\kappa(v)_{I},\quad  \text{for $\abs{I}=s$}.
\] 
Moreover, the above formula is always true for $s=1$.
\end{lemma}

\begin{proof}
We have that $\kappa(uv)=\kappa(u)\kappa(v)$ for $u,v\in F$. 
If $\kappa(u)_J=\kappa(v)_J=0$ for all $|J|<s$, then 
$\kappa(u)=1+\sum_{\abs{I}=s}\kappa(u)_{I}y_I$ up to higher-order terms, 
and similarly for $\kappa(v)$.  Then
\begin{equation}
\label{eq:kuv}
\kappa(uv)=\kappa(u)\kappa(v)=1+\sum_{\abs{I}=s}(\kappa(u)_{I}+
\kappa(v)_{I})y_I+\text{higher-order terms}.
\end{equation}
Therefore, $\kappa(uv)_i=\kappa(u)_i+\kappa(v)_i$, and so 
$\kappa(uv)_{I}=\kappa(u)_{I}+\kappa(v)_{I}$.
\end{proof}

\subsection{Truncating the Magnus expansions}
\label{subsec:truncate magnus}
Recall that we defined in \eqref{eq:magnus} truncations 
$M_k$ of the Magnus expansion $M$ of a free group $F$.  
In a similar manner, we can also define the truncations 
of the Magnus expansion $\kappa$ for 
any finitely generated group $G$.

\begin{lemma}
\label{lem: reduceMagnus}
For each $k\ge 1$, the following diagram commutes:
\begin{equation}
\label{formula:kappa}
\begin{gathered}
\xymatrixcolsep{11pt}
\xymatrix{   
  \Q F\ar[drr]^{\kappa}\ar[rr]^(.48){ M} 
  && \widehat{T}(F_{\Q})
  \ar[d]^{\widehat{T}(\pi)} \ar[rr]^{\gr_k\quad} 
  &&\gr_k( \widehat{T}(F_{\Q}))\ar[d]^{\gr_k(\widehat{T}(\pi))}
  \ar@{=}[r] &\bigotimes^k \Q^n\ar[d]^{\otimes^k \pi}\phantom{.}
  \\
   && \widehat{T}(G_{\Q}) \ar[rr]^{\gr_k\quad} 
   &&\gr_k( \widehat{T}(G_{\Q}))
    \ar@{=}[r] &
   \bigotimes^k\Q^b . }
\end{gathered}
\end{equation}
\end{lemma}

\begin{proof}
The triangle on the left  of diagram \eqref{formula:kappa} commutes, 
since it consists of ring morphisms, by the definition 
of the Magnus expansion for a group.  The morphisms in the two squares 
are homomorphisms between $\Q$-vector spaces. The squares commute, 
since $\pi$ is a linear map.
\end{proof}

In diagram \eqref{formula:kappa}, let us denote the composition of 
$\kappa$ and $\gr_k$ by $\kappa_k$. We then obtain the diagram
\begin{equation}
\label{eq:kappai}
\begin{gathered}
\xymatrixcolsep{40pt}
\xymatrix{
\Q{F} \ar^(.5){\kappa}[r] \ar^{\kappa_k}@/^20pt/[rr] &  
\widehat{T}(G_{\Q})   \ar^{\gr_k}[r] & 
 \gr_k(\widehat{T}(G_{\Q}))
}.  
\end{gathered}
\end{equation}
In particular, $\kappa_1(r)=\sum_{i=1}^b \kappa(r)_i y_i$ for $r\in F$. 
By Lemma \ref{lem:symmetry}, if $r\in [F,F]$, then
\begin{equation}
\label{eq:kappa2}
\kappa_2(r)=\sum_{1\leq i<j\leq b} \kappa(r)_{i,j}(y_iy_j-y_jy_i).
\end{equation}
Notice that $\kappa_2(r)$ is a primitive element in the Hopf algebra ${T}(G_{\Q})$,  
which corresponds to the 
element $\sum_{i<j} \kappa_{i,j}(r)[ y_i,y_j ]$ in the 
free Lie algebra $\Lie(G_{\Q})$.

The next lemma provides a close connection between the Magnus 
expansion $\kappa$ and the classical Magnus expansion $M$. 
 
\begin{lemma}
\label{lem:quasiMagnus}
Let $(a_{i,s})$ be the  $b\times n$ matrix associated to the linear map 
$\pi\colon F_{\Q}\rightarrow G_{\Q}$, and let 
$r\in F$ be an arbitrary element. Then, for each $1\leq i,j\leq b$, we have that
\begin{equation*}
\kappa(r)_i=\sum\limits_{s=1}^n a_{i,s}\varepsilon_{s}(r) \quad \textrm{ and } \quad
\kappa(r)_{i,j}=\sum\limits_{s,t=1}^n a_{i,s}a_{j,t}\varepsilon_{s,t}(r).
\end{equation*}
\end{lemma}
\begin{proof}
By assumption, $\pi(x_s)=\sum_{i=1}^ba_{i,s}y_i$. By 
Lemma \ref{lem: reduceMagnus} (for $k=1$), we have 
\[
\kappa_1(r)=\pi\circ M_1(r)=
\pi\left(\sum_{s=1}^n\varepsilon_s(r)x_s\right)=
\sum_{s=1}^n\sum_{i=1}^ba_{i,s}\varepsilon_s(r)y_i,
\]
which gives the claimed formula for $\kappa(r)_i$. By 
Lemma \ref{lem: reduceMagnus} (for $k=2$), we have
\[
\kappa_2(r)=\pi\otimes \pi \circ M_2(r)=\pi\otimes \pi \left(\sum_{s,t=1}^n 
\varepsilon_{s,t}(r)x_s\otimes x_t \right)=\sum_{s,t=1}^n\sum_{i,j=1}^b 
\varepsilon_{s,t}(r)a_{i,s}a_{j,t} y_i\otimes y_j,
\]
which gives the claimed formula for $\kappa(r)_{i,j}$.
\end{proof}

\section{Echelon presentations and cellular chain complexes}
\label{sect:ecpres}

In this section we associate to every finitely presented group $G$  
an ``echelon approximation", $G_e$, such that they have
isomorphic cohomology on their respective $2$-complexes.  

\subsection{Presentation $2$-complex}
\label{subsec:2complex}
We start with a brief review 
of the cellular chain complexes associated to a presentation 
$2$-complex of a group, following the exposition from 
\cite{Brown, Fe, Fenn-S, MS-tokyo, PS-tams}.   Let $G$ be a 
group with a finite presentation $P=\langle \bx\mid \br\rangle$, where 
$\mathbf{x}=\{x_1,\dots, x_n\}$ and $\mathbf{r}=\{r_1,\dots,r_m\}$. 
Then $G=F/R$, where $F$ is the free group on generating set $\bx$ and 
$R$ is the (free) subgroup of $F$ normally generated by the set $\br\subset F$.

Let $K_P$ be the $2$-complex associated to this presentation of $G$, 
consisting of a $0$-cell $e^0$, one-cells $\{e_1^1,\cdots,e_n^1\}$ 
corresponding to the generators, and two-cells $\{e_1^2, \dots, e_m^2\}$ 
corresponding to the relators.  The $2$-complex $K_P$ depends on the 
presentation $P$ for the group $G$. However, if the presentation is 
understood, we may also denote this $2$-complex by $K_G$.

The (integral) cellular chain complex $C_*=C_*(K_P;\Z)$ is 
of the form $C_2\xrightarrow{d_2} C_1 \xrightarrow{d_1} C_0$, 
where $C_j$ are the free abelian groups on the specified bases.  
Furthermore, $d_1=0$, while the matrix of the boundary map 
$d_2\colon C_2(K_P;\Z)\rightarrow C_1(K_P;\Z)$ is the $m\times n$ Jacobian 
matrix $J_P=(\varepsilon_i(r_k))$. 

Next, let $p\colon \widetilde{K_{P}}\to K_P$ be the universal cover of the 
presentation $2$-complex, and fix a lift $\tilde{e}^0$ of the basepoint $e^0$.   
The cells $e^i_j$ of $K_P$ lift to cells $\tilde{e}^i_j$ at the basepoint $\tilde{e}^0$.  
Let $\widetilde{C}_{*}=C_*(\widetilde{K_{P}};\Z)$ be the (equivariant) cellular chain 
complex of the universal cover.  This is a chain complex of free $\Z{G}$-modules 
of the form $\widetilde{C}_2\xrightarrow{\tilde{d}_2} 
\widetilde{C}_1 \xrightarrow{\tilde{d}_1} \widetilde{C}_0$, with $\widetilde{C}_0=\Z{G}$, 
$\widetilde{C}_1=(\Z{G})^n$ generated by the set $\{\tilde{e}^1_1,\dots, \tilde{e}^1_n \}$,  
and  $\widetilde{C}_2=(\Z{G})^m$  generated by the set $\{\tilde{e}^2_1,\dots, \tilde{e}^2_m \}$.  
The differentials in this chain complex are the $\Z{G}$-linear maps given by
\begin{equation}
\label{eq:differentials}
\tilde{d}_1(\tilde{e}^1_i)=x_i-1, \quad 
\tilde{d}_2(\tilde{e}^2_j)=\sum_{k=1}^m \varphi(\partial_k r_j) \tilde{e}^1_k, 
\end{equation}
where $\varphi\colon F\surj G$ is the presenting homomorphism for our group. 

\subsection{Echelon presentations}
\label{subsec:echelon}
We now introduce a special type of group presentations which will play 
an important role in the sequel.  

\begin{definition}
\label{def:echelon}
Let $G$ be a group with a finite presentation 
$P=\langle \bx\mid \bw\rangle$, where 
$\mathbf{x}=\{x_1,\dots, x_n\}$ and $\mathbf{w}=\{w_1,\dots,w_m\}$. 
We say $P$ is an \textit{echelon presentation}\/  
if the augmented Fox Jacobian matrix $(\varepsilon_i(w_k))$ 
is in row-echelon form.
\end{definition}

Let $K_G$ be the $2$-complex associated to the above 
presentation for $G$.  Suppose the pivot elements of  the 
$m\times n$ matrix $(\varepsilon_i(w_k))$ 
are in position $\{i_1,\dots, i_d\}$, and let $b=n-d$.
Since this matrix  is in row-echelon form, 
the vector space $H_1(K_G;\Q)=\Q^{b}$ has basis 
$\by=\{y_1,\dots, y_b\}$, where $y_j=e^1_{i_{d+j}}$ for $1\leq j\leq b$.
Furthermore, the vector space  $H_2(K_G;\Q)=\Q^{m-d}$ has basis 
$\{e^2_{d+1},\dots,  e^2_{m}\}$.
We will choose as basis for $H^1(K_{G};\Q)$ the set $\{u_1,\dots, u_b\}$, 
where $u_i$ is the Kronecker dual to $y_i$.   

\begin{remark}
\label{rem:commrel}
Suppose $G$ admits a commutator-relators presentation of 
the form $P=F/R$, with $R\subset [F,F]$.  Then the augmented 
Fox Jacobian matrix $(\varepsilon_i(r_k))$ is the zero matrix,   
and thus the presentation $P$ is an echelon presentation. In this case,
the integer (co)homology groups of $K_G$ are torsion-free, and so 
the aforementioned choices of bases work for integer (co)homology, 
as well.
\end{remark}

More generally, the next proposition shows that for any finitely presented group, 
we can construct a group with an echelon presentation such that the 
cohomology groups of the corresponding presentation $2$-complexes are isomorphic.

\begin{prop}
\label{prop: basis} 
Let $G$ be a finitely presented group.
There exists then a group $G_e$ with echelon presentation, 
and a map $f\colon K_{G_e}\to K_G$ between the respective 
presentation $2$-complexes such that the induced homomorphism 
on fundamental groups, $\rho=f_{\sharp} \colon G_e \surj G$, is  
surjective, and the induced homomorphism in cohomology, 
$f^*\colon H^*(K_G;\Z)\to  H^*(K_{G_e};\Z)$,  
in an isomorphism. 
\end{prop}

\begin{proof}
Suppose $G$ has presentation $\langle x_1,\dots, x_n\mid r_1,\dots,r_m\rangle$.
As in the above discussion, the matrix of the boundary map 
$d_2^*\colon C^1(K_G;\Z) \to C^2(K_G;\Z)$ 
is the transpose of the $m\times n$ Jacobian matrix 
$(\varepsilon_i(r_k))$. By Gaussian elimination over $\Z$, 
there exists a matrix $C=(c_{l,k})\in \GL(m;\Z)$ 
such that $C\cdot d_2^*$ is in row-echelon form (also 
known as Hermite normal form).  We define a new group,
\begin{equation}
\label{eq:Gtilde}
G_e=\langle x_1,\dots, x_n\mid w_1, \dots, w_m\rangle,
\end{equation}
by setting $w_k=r_1^{c_{1,k}}r_2^{c_{2,k}}\cdots r_m^{c_{m,k}}$ 
for $1\leq k\leq m$.  

Let $h\colon K^{(1)}_{G_e} \to K^{(1)}_G$ be the homeomorphism 
between the $1$-skeleta of the respective $2$-complexes obtained 
by matching $1$-cells. If 
$\psi_k\colon S^1\to K^{(1)}_{G_e}$ denotes the attaching map of the  
$2$-cell in $K_{G_e}$ corresponding to the relator $w_k$, then 
by construction $h\circ \psi_k$ is null-homotopic in $K_G$.  Thus, 
$h$ extends to a cellular map $f\colon K_{G_e} \to K_G$. 
Clearly, the induced homomorphism 
$\rho=f_{\sharp}\colon G_e\to G$ is surjective.
Furthermore, the map $f$ induces a chain map between the 
respective cellular chain complexes, $f_*\colon C_*(K_{G_e};\Z) \to 
C_*(K_G;\Z)$, with $f_2$ given by the matrix $C$. 
It is now straightforward to see that the map $f$ induces an isomorphism 
in homology, and thus, by the Universal Coefficients theorem, 
an isomorphism in cohomology, too.
\end{proof}

Note that the group $G_e$ constructed above depends on the 
given (finite) presentation for $G$, not just on the isomorphism type of $G$.  
On the other hand, if $G$ is a commutator-relators group, then, 
by Remark \ref{rem:commrel}, the group 
$G_e$ is isomorphic to $G$.

\subsection{A transferred basis} 
\label{subsec:transferred basis}

Once again, let $G$ be a group admitting a finite presentation 
$\langle \bx \mid \br\rangle$, where $\bx=\{x_1,\dots, x_n\}$ and 
$\br=\{r_1,\dots, r_m\}$, and let $K_G$ be the corresponding 
presentation $2$-complex.   Using an echelon approximation 
for the given presentation, we describe now convenient bases for 
the $\Q$-vector spaces $H^1(K_G,\Q)$ and  $H^2(K_G,\Q)$, 
which will be used extensively in the next two sections.  

By Proposition \ref{prop: basis}, there exists a group $G_e$ with 
echelon presentation $\langle \bx \mid \bw \rangle$, where 
$\bw=\{w_1,\dots, w_m\}$, and a map $f\colon K_{G_e}\to K_G$ 
inducing an isomorphism in (co)homology.
As in \S\ref{subsec:echelon}, we may choose a basis 
$\by=\{y_1,\dots, y_b\}$ for the $\Q$-vector space $H_1(K_G;\Q)\cong H_1(K_{G_e};\Q)$;  
let $\{u_{1},\dots, u_b\}$ be the dual basis for $H^1(K_G;\Q)\cong H^1(K_{G_e};\Q)$. 
We also choose a basis $\{z_1,\dots, z_m\}$ for $C_2(K_G;\Q)$ and a basis 
$\{e^2_1,\dots, e^2_m\}$ for $C_2(K_{G_e};\Q)$ corresponding to 
$\{1\otimes \tilde{e}_1^2, \dots, 1\otimes \tilde{e}_m^2\}$. 
Finally, if we set
\begin{equation}
\label{eq:gammak}
\gamma_k:=f_*(e_k^2)=\sum_{l=1}^m c_{l,k} z_l \,  , 
\end{equation}
then $\{\gamma_1, \dots \gamma_m\}$ 
is another basis for $C_2(K_G;\Q)$. Furthermore, $\{e_{d+1}^2,\dots,e_m^2\}$ 
is a basis for $H_2(K_{G_e};\Q)$ and $\{\gamma_{d+1},\dots, \gamma_m\}$ 
is a basis for $H_2(K_G;\Q)$.
Thus, $H^2(K_G;\Q)$ has dual basis $\{\beta_{d+1},\dots, \beta_m\}$.

\section{Group presentations and (co)homology}
\label{sect:pres}

We compute in this section the cup-product in 
the cohomology ring of the $2$-complex of a finitely 
presented group in terms of the 
Magnus expansion associated to the presentation. 
 
\subsection{A chain transformation}
\label{subsec:chaintrans}
We start by reviewing the classical bar construction. 
Let $G$ be a discrete group, and let $B_*(G)$ be the normalized 
bar resolution (see e.g.~\cite{Brown, Fenn-S}), 
where $B_p(G)$ is the free left $\Z G$-module on generators 
$[g_1|\dots|g_p]$, with $g_i\in G$ and 
$g_i\neq 1$, and $B_0(G)=\Z{G}$ is free on one generator, $[\ ]$.
The boundary operators are $G$-module homomorphisms, 
$\delta_p\colon B_p(G)\rightarrow B_{p-1}(G)$, defined by
\begin{equation}
\label{eq:bar diff}
\delta_p[g_1|\dots|g_p]=g_1[g_2|\dots|g_p]
+\sum\limits_{i=1}^{p-1}(-1)^i[g_1|\dots|g_ig_{i+1}|\dots|g_p]
+(-1)^p[g_1|\dots|g_{p-1}].
\end{equation}
In particular, $\delta_1[g]=(g-1)[\ ]$ and $\delta_2[g_1|g_2]=g_1[g_2]-[g_1g_2]+[g_1]$. 
Let $\varepsilon\colon B_0(G)\to\Z$ be the augmentation map.  
We then have a free resolution $B_{*}(G)\xrightarrow{\varepsilon} \Z$ 
of the trivial $G$-module $\Z$.

We view here $\Z$ as a right $\Z{G}$-module, with action 
induced by the augmentation map. 
An element of the cochain group $B^p(G)=\hom_{\Z{G}}(B_p(G),\Z)$ 
may be viewed as a set function $u\colon G^p\to \Z$ 
satisfying the normalization condition $u(g_1,\dots,g_p)=0$ if some $g_i=1$.
The cup-product of two $1$-dimensional classes 
$u,u^{\prime}\in H^1(G;\Z)\cong \hom(G,\Z)$ 
is given by
\begin{equation}
\label{eq:cup}
u\cup u^{\prime}[g_1|g_2]=u(g_1)u^{\prime}(g_2).
\end{equation}

For future use, we record a result due to Fenn and Sjerve (\cite[Thm.~ 2.1 and p.~327]{Fenn-S}).

\begin{lemma}[\cite{Fenn-S}]
\label{lem:fenn}
There exists a chain transformation 
$T\colon C_*(\widetilde{K_{G}})\rightarrow B_*(G)$ 
of augmented chain complexes, 
\begin{equation*}
\xymatrix{
0 & \Z\ar[l]\ar@{=}[d] & C_0(\widetilde{K_{G}})\ar[l]_{\varepsilon}\ar[d]^{T_0} 
&C_1(\widetilde{K_{G}})\ar[l]_{\tilde{d}_1}\ar[d]^{T_1} 
& C_2(\widetilde{K_{G}})\ar[l]_{\tilde{d}_2}\ar[d]^{T_2} &0\ar[l]\ar[d] & \cdots \phantom{.}\ar[l] \\
0 & \Z\ar[l] & B_0(G)\ar[l]_{\varepsilon}
&B_1(G)\ar[l]_{\delta_1} & B_2(G)\ar[l]_{\delta_2} 
&B_3(G)\ar[l]& \cdots . \ar[l]  }
\end{equation*}
defined by  $T_0(\lambda):=\lambda[~]$, 
\begin{equation}
\label{eq:transformation}
T_1(\tilde{e}_i^1)= [x_i] \quad and \quad 
T_2(\tilde{e}_k^2)=\tau_1T_1\tilde{d}_2(\tilde{e}^2_k),
\end{equation}
where $\tau_1\colon B_1(G)\to B_2(G)$
is the homomorphism defined by
\begin{equation}
\label{eq:tau}
\tau_1(g[g_1])=[g|g_1],
\end{equation}
for all $g,g_1\in G$.
\end{lemma}

\subsection{Cup products for echelon presentations} 
\label{subsec:cup gtilde}

Now let $G$ be a group with echelon presentation 
$G=\langle \bx \mid \bw \rangle$, where $\bx=\{x_1,\dots, x_n\}$ and 
$\bw=\{w_1,\dots, w_m\}$, as in Definition \ref{def:echelon}. 
We let $B_*(G;\Q)=\Q\otimes B_*(G)$ and $B^*(G;\Q)=\Q\otimes B^*(G)$.

\begin{lemma} 
\label{lem:uk}
For each basis element 
$u_i\in H^1(K_{G};\Q)\cong H^1(G;\Q)$ 
as above, and each $r\in F$, we have that 
\[
u_i([\varphi(r)])=\sum_{s=1}^n\varepsilon_s(r)a_{i,s}=\kappa_i(r), 
\]
where $(a_{i,s})$ is the $b\times n$ matrix for the projection 
map $\pi\colon F_{\Q}\to G_{\Q}$.
\end{lemma}

\begin{proof}
If $r\in F$, then $\varphi(r)\in G$ and $[\varphi(r)]\in B_1(G)$.  
Hence, 
\begin{equation}
\label{eq:ui}
u_i([\varphi(r)])=\sum_{s=1}^n \varepsilon_s(r) u_i([x_s])
=\sum_{s=1}^n\varepsilon_{s}(r)a_{i,s} 
=\kappa_i(r).
\end{equation}

Since $H^1(G;\Q)\cong B^1(G;\Q)\cong \hom(G,\Q)$, we may 
view $u_i$ as a group homomorphism. This yields the first equality 
in \eqref{eq:ui}. Since $\pi(x_s)=\sum_{j=1}^b a_{i,s} y_i$ and $u_i=y_i^*$, 
the second equality follows. The last equality follows from Lemma \ref{lem:quasiMagnus}.
\end{proof}

\begin{theorem}
\label{thm:cup G tilde}
Let $G$ be a group with echelon presentation 
$G=\langle \bx \mid \bw \rangle$.
The cup-product map $\mu_{K_G}\colon H^1(K_{G};\Q)\wedge 
H^1(K_{G};\Q)\rightarrow H^2(K_{G};\Q)$ 
is given by
$(u_i\cup u_j,  e^2_k)=\kappa( w_k)_{i,j}$, 
for $1\leq i,j\leq b$ and $d+1\leq k\leq m$, where $\kappa$ is the 
Magnus expansion of $G$.
\end{theorem}

\begin{proof}
Let us write the Fox derivative $\partial_t( w_k)$ as a finite sum, 
\begin{equation}
\label{eq:foxderivative}
\partial_t( w_k)=\sum_{x\in F} p_{tk}^x x,
\end{equation}
for $1\leq t\leq n$, and $1\leq k\leq m$. 
We then have
\begin{align}
T_2(\tilde{e}_k^2)&=\tau_1T_1 (\tilde{d}_2(\tilde{e}_k^2))
& \text{by~\eqref{eq:transformation}}\notag \\
&=\tau_1T_1\left(\varphi(\partial_1(w_k)),\dots,\varphi(\partial_n(w_k)) \right)
& \text{by~\eqref{eq:differentials}}\label{eq:t2}\\
&=\tau_1\Big(\sum_{t=1}^n \varphi\big(\partial_t(w_k)\big)[x_t]\Big) 
& \text{by~ \eqref{eq:transformation}}\notag\\
&=\sum_{t=1}^n\sum_{x\in F} p_{tk}^x[\varphi(x)|x_t].
& \text{by~\eqref{eq:tau}}\notag
\end{align}

The chain transformation 
$T\colon C_*(\widetilde{K_{G}};\Q)\rightarrow B_*(G;\Q)$ 
induces an isomorphism on first cohomology, 
$T^*\colon H^1(G;\Q)\rightarrow H^1(K_{G};\Q)$. 
Let us view $u_i$ and $u_j$ as elements in $H^1(G;\Q)$. 
We then have
\begin{align*}
 (u_i\cup u_j, 1\otimes \tilde{e}_k^2)
 &=(u_i\cup u_j, 1\otimes T_2(\tilde{e}_k^2)) & \\
 &=(u_i\cup u_j, \sum_{t=1}^n\sum_{x\in F} p_{tk}^x[\varphi(x)|x_t])
 &  \text{by \eqref{eq:t2}}\\
 &=\sum_{t=1}^n\sum_{x\in F} p_{tk}^xu_i(\varphi(x))u_j(x_t)
 & \text{by \eqref{eq:cup}}\\
 &=\sum_{t=1}^n\sum_{x\in F} p_{tk}^xu_i(\varphi(x))a_{j,t} 
 & \text{by~ Lemma~ \ref{lem:uk}}\\
 &=\sum_{t=1}^n\sum_{x\in F} p_{tk}^x\sum_{s=1}^n a_{i,s}\varepsilon_s(x)a_{j,t}  
 & \text{by~ Lemma~ \ref{lem:uk}}\\
 &=\sum_{t=1}^n\sum_{s=1}^n  \left(a_{j,t}a_{i,s}\varepsilon_{s,t}(w_k)  \right)
 & \text{by \eqref{eq:foxderivative}}\\
 &=\kappa(w_k)_{i,j}, & \text{by~ Lemma~ \ref{lem:quasiMagnus}}
\end{align*}
and this completes the proof.
\end{proof}

\subsection{Cup products for finite presentations}
\label{subsec:cupg}
Let $G$ be a group with a finite presentation 
$\langle \bx \mid \br\rangle$.
By Proposition \ref{prop: basis}, there exists a group $G_e$ with 
echelon presentation $\langle \bx\mid \bw\rangle$, and a map 
$f\colon K_{G_e}\to K_G$  inducing an isomorphism in  cohomology.
Using the bases for $H^*(K_G;\Q)$ transferred from suitable bases 
for $H^*(K_{G_e};\Q)$ as in \S\ref{subsec:transferred basis}, 
we obtain an explicit formula for computing cup-products 
in the rational cohomology of the presentation $2$-complex $K_G$.
 
\begin{theorem}
\label{Thmcupproduct again}
In the aforementioned bases for $H^*(K_G;\Q)$, the cup-product map 
$\mu_{K_G}\colon H^1(K_G;\Q)\wedge H^1(K_G;\Q)\rightarrow 
H^2(K_G;\Q)$ is given by 
\[
u_i \cup u_j=\sum_{k=d+1}^{m}  \kappa(w_k)_{i,j} \beta_k.
\]
That is, $(u_i\cup u_j, \gamma_k)=\kappa(w_k)_{i,j}$, 
for all $1\leq i,j\leq b $.
\end{theorem}

\begin{proof}
As in the discussion from \S\ref{subsec:transferred basis}, the elements 
$\gamma_k=f_*(w_{k})$ with $d< k \leq m$ form a basis for $H_2(K_G;\Q)$.   
Hence,
\begin{align*}
(u_i\cup u_j, \gamma_k)=\left(u_i\cup u_j, f_*(e_k^2)\right)  
&=\left(f^*(u_i\cup u_j), e_k^2\right)& \\
&=\left(u_i\cup u_j, e_k^2\right) & \text{since~$f^*(u_i)=u_i$} \\
 &=\kappa(w_k)_{i,j}  & \text{by~Theorem~\ref{thm:cup G tilde}}.
\end{align*}
The claim follows. 
\end{proof}

Let us consider now in more detail the case when the group $G$ is a 
commutator-relators group.  In that case, as noted in \S\ref{subsec:kappa},  
the Magnus expansion $\kappa=\kappa_G$ coincides with the classical 
Magnus expansion $M$.  Furthermore, by Remark \ref{rem:commrel} 
both $H_*(K_G;\Z)$ and $H^*(K_G;\Z)$ are torsion-free, and the 
aforementioned rational bases are also integral bases for these 
free $\Z$-modules.  Moreover, we may take $G_e=G$, and note that 
all the arguments from this section work over $\Z$ in this case.  
Using these observations, and the fact that $M(r_k)_{i,j}=\varepsilon_{ij}(r_k)$, 
we recover as a corollary the following result of Fenn and Sjerve \cite{Fenn-S}.

\begin{corollary}[\cite{Fenn-S}, Thm.~2.4]
\label{cor:fs}
For a commutator-relators group $G=\langle \bx \mid \br\rangle$, 
the cup-product map $\mu_K\colon H^1(K_{G};\Z)\wedge 
H^1(K_{G};\Z)\rightarrow H^2(K_{G};\Z)$ 
is given by
$(u_i\cup u_j, e_k^2)=\varepsilon_{ij}(r_k)$,
for $1\leq i,j\leq n$ and $1\leq k\leq m$.
\end{corollary}

\section{A presentation for the holonomy Lie algebra}
\label{sect:holo lie pres}

In this section, we give presentations for the holonomy Lie algebra 
of a finitely presented group, and for the solvable quotients of this 
Lie algebra.  In the process, we complete the proof of 
Theorem \ref{thm:intro} from the Introduction. 

\subsection{The holonomy Lie algebra of a group}
\label{subsec:holo lie gp}

We start by reviewing the construction of the holonomy Lie algebra 
of a finitely generated group $G$, following 
\cite{Chen73, Kohno, Markl-Papadima, PS-imrn, SW-formality}.   
Set
\begin{equation}
\label{eq:holog}
\h(G)=\Lie(H_1(G;\Q))/\langle\im \mu^{\vee}_G\rangle,
\end{equation}
where $\mu^{\vee}_G$ is the dual to the cup-product map 
$\mu_G\colon H^1(G;\Q)\wedge H^1(G;\Q)\to H^2(G;\Q)$. 
If $\varphi \colon G_1\to G_2$ is a group homomorphism, then 
the induced homomorphism in cohomology, $\varphi^*\colon H^1(G_2;\Q)\to H^1(G_1;\Q)$,  
yields a morphism of graded Lie algebras, $\h(\varphi)\colon \fh(G_1)\to \fh(G_2)$.  
Moreover, if $\varphi$ is surjective, then $\h(\varphi)$ is also surjective. 

In the definition of the holonomy Lie algebra of $G$, we used the 
cohomology ring of a classifying space $K(G,1)$.  More generally, 
if $X$ is a connected space with $b_1(X)<\infty$, we may define 
its holonomy Lie algebra as $\h(X)=\Lie(H_1(X;\Q))/\langle\im \mu^{\vee}_X\rangle$.  
As above, a continuous map $f\colon X\to Y$ induces a morphism 
$\h(f)\colon \h(X)\to \h(Y)$; moreover, if $f\simeq  g$, then $\h(f)=\h(g)$. 
The proof of the next lemma is straightforward. 

\begin{lemma}
\label{lem:holo def}
Let $f\colon X\to Y$ be a map between connected spaces with finite first Betti 
numbers.  Suppose $f$ induces isomorphisms in rational cohomology in degrees 
$1$ and $2$.  Then the map $\h(f)\colon \fh(X)\to \fh(Y)$ is an isomorphism.
\end{lemma}

In definition \eqref{eq:holog}, we may replace the classifying space $K(G,1)$ 
used to compute group cohomology by any other connected 
CW-complex with the same fundamental group. 
The next lemma, which slightly improves on a result from 
\cite{PS-imrn, SW-formality}, makes this more precise.

\begin{lemma}
\label{lem:holo inv}
Let $G$ be a finitely generated group, and 
let $X$ be a connected CW-complex with $\pi_1(X)=G$. 
There is then a natural isomorphism $\h(X) \isom \h(G)$.   
\end{lemma}

\begin{proof}
We may construct a classifying space for the group $G$ by attaching cells of 
dimension $3$ and higher to the space $X$. The inclusion 
map, $j=j_X\colon X\to K(G,1)$, induces a map on cohomology rings, 
$j^*\colon H^*(K(G,1);\Q)\to H^*(X;\Q)$, which is an isomorphism 
in degree $1$ and an injection in degree $2$. In particular, $b_1(X)=b_1(G)<\infty$. 
Furthermore, $j^2$ restricts to an isomorphism from 
$\im(\mu_G)$ to $\im(\mu_{X})$.  Taking duals, we obtain 
the following diagram. 
\begin{equation}
\label{eq:partial}
\begin{gathered}
\xymatrix{
H_2(X;\Q)\ar@{->>}[d]^{j_*}\ar@/^18pt/[rr]^{\mu^{\vee}_X} \ar@{->>}[r] 
& \im(\mu^{\vee}_X)\ar[d]_{\cong}^{j_*}\ar@{^(->}[r] & H_1(X;\Q)\otimes H_1(X;\Q)\ar[d]_{\cong}^{j_*\otimes j_*} \\
H_2(G;\Q) \ar@{->>}[r]\ar@/_22pt/[rr]^{\quad\mu^{\vee}_G} & \im(\mu^{\vee}_G)\ar@{^(->}[r] & H_1(G;\Q)\otimes H_1(G;\Q)
}
\end{gathered}
\end{equation}
\vskip4pt

Hence, the isomorphism $j_*\otimes j_*$ identifies 
$\im(\mu^{\vee}_{X})$ with $\im(\mu^{\vee}_G)$.  Thus, the extension 
to free Lie algebras of the isomorphism $j_*\colon H_1(X;\Q)\to H_1(G;\Q)$ 
factors through an isomorphism $\h(j)\colon \h(X)\to \h(G)$.  

To show that this isomorphism is natural, let $f\colon X\to Y$ be a map of 
pointed, connected CW-complexes with finitely generated fundamental 
groups, and let $g\colon K(\pi_1(X),1)\to K(\pi_1(Y),1)$ 
be the map (unique up to homotopy) induced 
by the homomorphism $f_{\#}\colon \pi_1(X)\to \pi_1(Y)$.  
Then $g \circ j_X\simeq  j_Y\circ f$, and thus 
$\h(g) \circ \h(j_X)=  \h(j_Y) \circ\h(f)$.
\end{proof}

Putting together the previous two lemmas, we obtain the following corollary.

\begin{corollary}
\label{cor:holo map}
Let $G_1$ and $G_2$ be two finitely generated groups, with 
presentation $2$-complexes $K_1$ and $K_2$. 
Let $f\colon K_1\to K_2$ be a cellular map, and let 
$\varphi=f_{\sharp}\colon G_1\to G_2$ be the induced 
homomorphism. 
If $f^*\colon H^*(K_{2},\Q) \to H^*(K_{1},\Q)$ is an isomorphism,    
then $\h(\varphi)\colon \fh(G_1)\to \fh(G_2)$ 
is also an isomorphism.
\end{corollary}

Next, we show that, if need be, the group $G$ we started with may 
be replaced by a finitely presented group with the same holonomy Lie algebra.

\begin{prop}
\label{prop:holo fp}
Let $G$ be a finitely generated group. There exists then a 
finitely presented group $G_f$ and a homomorphism 
$G_f \to G$ inducing an isomorphism $\h(G_f)\isom \h(G)$. 
\end{prop}

\begin{proof}
Let $X$ be a connected CW-complex with $\pi_1(X)=G$.  
Since $G$ is finitely generated, we may assume $X$ has 
finitely many $1$-cells.  
The proof of Proposition 4.1 from \cite{PS-mrl} shows that 
there exists a connected, finite subcomplex $Z$ of $X$ such that 
the inclusion $Z\to X$ induces isomorphisms 
$H_1(Z;\Q)\cong H_1(X;\Q)$ and $\im(\mu^{\vee}_Z)\cong \im(\mu^{\vee}_X)$.
Consequently, $\fh(Z)\cong \fh(X)$. 
Letting $G_f=\pi_1(Z)$, the claim follows 
from Lemma \ref{lem:holo inv}.
\end{proof}

\subsection{Magnus expansion and holonomy}
\label{subsec:mpres}
Let $G$ be a group admitting a finite presentation, $P=\langle \bx\mid \br\rangle$. 
As shown in Proposition \ref{prop: basis}, there exists a group $G_e$ with 
echelon presentation $P_e=\langle \bx\mid \bw\rangle$, and a map 
$f\colon K_{G_e}\to K_G$  inducing an isomorphism in  cohomology 
and an epimorphism $\rho\colon G_e\surj G$ on 
fundamental groups.  By Corollary \ref{cor:holo map}, the induced map 
between the respective holonomy Lie algebras, $\fh(\rho)\colon \fh(G_e)\isom \fh(G)$, 
is an isomorphism.  

So let us consider a group $G=F/R$ admitting an echelon presentation 
$P=\langle \bx\mid \bw\rangle$, where $\bx=\{x_1,\dots, x_n\}$ and 
$\bw=\{w_1,\dots, w_m\}$.   We now give a more explicit presentation 
for the holonomy Lie algebra $\fh(G)$.  

Let $\partial_i(w_k)\in \Z{F}$ be the Fox derivatives of the relations,   
and let $\varepsilon_i(w_k)=\varepsilon( \partial_i(w_k) )\in \Z$ be their augmentations.  
Recall from \S \ref{subsec:cup gtilde} that we can choose a basis 
$\by=\{y_1,\dots,y_b\}$ for $H_1(K_P;\Q)$ and a basis 
$\{e^2_{d+1},\dots, e^2_{m}\}$ for  $H_2(K_P;\Q)$, where $d$ 
is the rank of Jacobian matrix $J_P=(\varepsilon_i(w_k))$, viewed as 
an $m\times n$ matrix over $\Q$. 
Let $\Lie(\by)$ be the free Lie algebra 
over $\Q$ generated by $\by$ in degree $1$.
Recall that $\kappa_2$ is the degree $2$ part of the Magnus expansion of $G$
given explicitly in \eqref{eq:kappa2}. Thus, we can identify 
$\kappa_2(w_k)$ 
with $\sum_{i<j} \kappa(w_k)_{i,j}[y_i,y_j]$ in 
$\Lie(\by)$ for $d+1\leq k\leq m$. 
 
\begin{theorem}
\label{ThmholonomyLie}
Let $G$ be a group admitting an echelon presentation 
$P=\langle \bx \mid \bw\rangle$. 
Then there exists an isomorphism of graded Lie algebras
\[
\xymatrixcolsep{18pt}
\xymatrix{
\fh(G) \ar^(.2){\cong}[r] & \left.\Lie(\by)
\middle/{\rm ideal} (\kappa_2(w_{d+1}),\dots,\kappa_2(w_{m}))\right.} .
\]
\end{theorem}

\begin{proof}
Combining Theorem \ref{thm:cup G tilde} with the fact that 
$(u_i\wedge u_j, \mu^{\vee}(e^2_k))=(\mu(u_i\wedge u_j),e^2_k)$, we see that 
the dual cup-product map, $\mu^{\vee}\colon H_2(K_P;\Q)\rightarrow 
H_1(K_P;\Q)\wedge H_1(K_P;\Q)$, is given by
\begin{equation}
\label{eq:cup star}
\mu^{\vee}(e^2_k)=\sum_{1\leq i<j\leq b}\kappa(w_k)_{i,j} (y_i\wedge y_j) .
\end{equation}
Hence, the following diagram commutes,
\begin{equation}
\label{eq:cd h2}
\begin{gathered}
\xymatrix{
H_2(K_P;\Q)\ar[r]^(.35){\mu^{\vee}}\ar@{^(->}[d]& H_1(K_P;\Q)\wedge H_1(K_P;\Q)\ar@{^(->}[d]\:\\
C_2(K_P;\Q)\ar[r]^(.35){\kappa_2} & H_1(K_P;\Q)\otimes H_1(K_P;\Q)\, .
}
\end{gathered}
\end{equation}

Using now the identification of $\kappa_2(w_k)$ and 
$\sum_{i<j} \kappa(w_k)_{i,j}[y_i,y_j]$ as elements of $\Lie(\by)$,
the definition of the holonomy Lie algebra, 
and the fact that $\fh(G)\cong \fh(K_P)$, we 
arrive at the desired conclusion.
\end{proof}

\begin{corollary}
\label{cor:Uh}
The universal enveloping algebra of $\fh(G)$ has presentation
\[
U(\fh(G))=\Q\langle \by \rangle/\ideal(\kappa_2(w_{n-b+1}),\dots, \kappa_2(w_m)).
\]
\end{corollary}

If $G=\langle \bx \mid \br\rangle$ is a commutator-relators group, then
the group $H_1(K_{G};\Z)$ is torsion-free, and thus 
the integer holonomy Lie ring $\fh(G;\Z)$ can be defined as in \eqref{eq:holog}, 
using integer (co)homology, instead, see \cite{Markl-Papadima, PS-imrn} for details.  
Furthermore, as in \S\ref{subsec:magnus}, for each $r\in [F,F]$, the primitive 
element $M_2(r)\in \widehat{T}_2(F_{\ab})$ corresponds to the element  
$\sum_{i<j} \varepsilon_{i,j}(r)[x_i,x_j]$ from the degree~$2$ piece of the free 
Lie ring $\Lie_{\Z}(\bx)$.  Using this observation and Corollary \ref{cor:fs}, 
we recover a result from \cite{PS-imrn}.
 
\begin{corollary}[\cite{PS-imrn}, Prop.~7.2]
\label{Cor:holo commrel} 
If $G=\langle x_1,\dots ,x_n \mid \br\rangle$ is a commutator-relators group, then  
\[
\fh(G;\Z)=\Lie_{\Z}(\bx)/{\rm ideal}\bigg\{\sum_{1\le i<j\le n} \varepsilon_{i,j}(r)[x_i,x_j]\: 
\big| \:r\in \br \big.\bigg\}\, .
 \]
\end{corollary}

\begin{prop}
\label{prop:univ quad}
For every quadratic Lie algebra $\fg$ over $\Q$, there exists a 
commutator-relators group $G_c$ such that $\fh(G_c)= \fg$.
\end{prop}

\begin{proof}
We use an approach similar to the proof of \cite[Prop.~6.2]{PS-crelle}. 
By assumption, we may write $\fg=\Lie(\bx)/\mathfrak{a}$,  
where $\bx=\{x_1,\dots, x_n\}$ and $\mathfrak{a}$ is an ideal generated by 
elements of the form $\ell_k=\sum_{1\leq i<j\leq n} c_{ijk}[x_i,x_j]$ for $1\leq k\leq m$, 
and where the coefficients $c_{ijk}$ are in $\Q$.  
Clearing denominators, we may assume all $c_{ijk}$ are integers. 
We can then define words $r_k=\prod_{1\leq i<j\leq n}[x_i,x_j]^{c_{ijk}}$ 
in the free group generated by $\bx$, and set 
$G_c=\langle \bx \mid r_1,\dots,r_m\rangle$. Clearly, $\varepsilon_{i,j}(r_k)=c_{ijk}$. 
The desired conclusion follows from Corollary \ref{Cor:holo commrel}.
\end{proof}

\begin{corollary}
\label{cor:gcomm}
For every finitely generated group $G$, there exists a 
commutator-relators group $G_c$ such that $\fh(G_c)= \fh(G)$.
\end{corollary}

\begin{proof}
From Proposition \ref{prop:holo fp}, the holonomy Lie algebra $\h(G)$ 
has a quadratic presentation.
Letting $\g=\h(G)$ and applying Proposition \ref{prop:univ quad} 
yields the desired conclusion.
\end{proof}

\subsection{Solvable quotients of holonomy Lie algebras}
\label{subsec:chen lie bis}

The next lemma follows straight from the definitions, using 
the standard isomorphism theorems. 

\begin{lemma}
\label{lem:DerivedLieIso}
Let $\fg=\Lie(V)/\mathfrak{r}$ be a finitely generated Lie algebra. Then 
$\fg/\fg^{(i)}\cong \Lie(V)/(\mathfrak{r}+\Lie(V)^{(i)})$. Furthermore, if 
$\mathfrak{r}$ is a homogeneous ideal, then this is an isomorphism of 
graded Lie algebras. 
\end{lemma}

The next result  sharpens and extends the first part of 
Theorem 7.3 from \cite{PS-imrn}. 

\begin{theorem}
\label{thm:holo chen lie}
Let $G=\langle \bx\mid \br\rangle$ be a finitely presented group, 
and set $\h=\h(G)$.  Let $\by=\{y_1,\dots,y_b\}$ be a basis of $H_1(G;\Q)$.
Then, for each $i\ge 2$, 
\[
\fh/\fh^{(i)} \cong \Lie(\by)/(  
{\rm ideal} (\kappa_2(w_{n-b+1}),\dots,\kappa_2(w_{m}))+ \Lie^{(i)}(\by)),
\]
where $b=b_1(G)$ and $w_k$ is defined in \eqref{eq:Gtilde}.
\end{theorem}
\begin{proof}
By Theorem \ref{ThmholonomyLie}, 
the holonomy Lie algebra $\fh$ is isomorphic to the quotient of the free Lie algebra 
$\Lie(\by)$ by the ideal generated by  $\kappa_2(w_{n-b+1}),\dots,\kappa_2(w_{m})$. 
The claim follows from Lemma \ref{lem:DerivedLieIso}.
\end{proof}

Using Corollary \ref{Cor:holo commrel},  we obtain the following consequence.

\begin{corollary}
\label{cor:comm rel chen}
Let $G=\langle x_1,\dots, x_n \mid r_1,\dots, r_m\rangle$ be a 
commutator-relators group, and let $\h=\h(G)$.
Then, for each $i\ge 2$, the Lie algebra $\fh/\fh^{(i)}$ is isomorphic to the 
quotient of the free Lie algebra $\Lie(\bx)$ by the sum of the ideals 
$M_2(r_1), \dots,M_2(r_m)$, and $\Lie^{(i)}(\bx)$.
\end{corollary}

\section{Lower central series and the holonomy Lie algebra}
\label{sect:lcs holo}

\subsection{Lower central series}
\label{subsec:lcs}

Let $G$ be a finitely generated group, and let $\{\Gamma_k G\}_{k\geq 1}$ 
be its lower central series (LCS). 
The LCS quotients of $G$ are finitely generated abelian groups.  
Taking the direct sum of these groups, we obtain a graded Lie ring over $\Z$,
\begin{equation}
\label{eq:grg}
\gr(G;\Z)=\bigoplus\limits_{k\geq 1}\Gamma_kG/\Gamma_{k+1}G. 
\end{equation}

The Lie bracket $[x,y]$ on $\gr(G;\Z)$ is induced from the group commutator,  
$[x,y]=xyx^{-1}y^{-1}$. More precisely, if $x\in \Gamma_r G$ and $y\in\Gamma_s G$, 
then $[x+\Gamma_{r+1}G,y+\Gamma_{s+1}G]=xyx^{-1}y^{-1}+\Gamma_{r+s+1}G$. 
The Lie algebra $\gr(G;\Q)=\gr(G;\Z)\otimes \Q$ is called the 
\textit{associated graded Lie algebra}\/ (over $\Q$) of the group $G$.  
For simplicity, we will usually drop the $\Q$-coefficients, and simply write 
it as $\gr(G)$.  

The group $G$ is said to be {\em nilpotent}\/ of class $\leq k$ if $\Gamma_{k+1} G=\{1\}$.
For each $k\ge 2$, the factor group $G/\Gamma_k G$ is the 
maximal $(k-1)$-step nilpotent quotient of $G$. 
The canonical projection $G\to G/\Gamma_k G$ induces an epimorphism 
$\gr(G) \to \gr(G/\Gamma_k G)$, which is an isomorphism in degrees 
$s< k$. 
We refer to Lazard \cite{Lazard54} and Magnus et al.~\cite{Magnus-K-S} for more details..

\subsection{A comparison map}
\label{subsec:Phimap}

Once again, let $G$ be a finitely generated group.   Although the 
next lemma is known, we provide a proof, both for the sake of completeness, 
and for later use. 

\begin{lemma}[\cite{Markl-Papadima, PS-imrn}]
 \label{lem:holoepi}
There exists a natural epimorphism of graded $\Q$-Lie algebras, 
\begin{equation*}
\label{map:Phi}
\xymatrixcolsep{20pt}
\xymatrix{\Phi_G\colon \fh(G) \ar@{->>}[r]& \gr(G)}, 
\end{equation*}
inducing isomorphisms in degrees $1$ and $2$.  
\end{lemma}

\begin{proof}
As first noted by Sullivan \cite{Sullivan75} in a particular case, 
and then proved by Lambe \cite{Lambe86} in general, there is a 
natural exact sequence 
\begin{equation}
\label{eq:sl}
\xymatrixcolsep{20pt}
\xymatrix{0\ar[r]& (\Gamma_2G/\Gamma_3 G\otimes \Q)^* 
\ar^(.45){\beta}[r]& H^1(G;\Q)\wedge H^1(G;\Q) \ar^(.62){\mu_G}[r]& H^2(G;\Q)},  
\end{equation}
where $\beta$ is the dual of Lie bracket product. Consequently, 
$\im(\mu^{\vee}_G) = \ker(\beta^{\vee})$. 

Recall now that the associated graded Lie algebra $\gr(G)$ 
is generated by its degree $1$ piece,  $H_1(G;\Q)\cong \gr_1(G)$.  
Hence, there is a natural epimorphism of graded $\Q$-Lie algebras,
\begin{equation}
\label{eq:phig}
\xymatrixcolsep{20pt}
\xymatrix{\varphi_G\colon \Lie(H_1(G;\Q))\ar@{->>}[r]& \gr(G)}, 
\end{equation}
restricting to the identity in degree $1$, and to the Lie bracket map 
$[\,,\,]\colon \bigwedge^2 \gr_1(G)\rightarrow \gr_2(G)$ 
in degree $2$. By the above observation, the kernel of 
this map coincides with the image of $\mu^{\vee}_G$.  
Thus, $\varphi_G$ factors through a morphism $\Phi_G\colon \h(G)\to \gr(G)$, 
which enjoys all the claimed properties.
\end{proof} 

\subsection{Nilpotent and derived quotients}
\label{subsec:quotients}

As a quick application, let us compare the holonomy Lie algebra 
of a group to the holonomy Lie algebras of its nilpotent quotients and derived quotients. 

\begin{prop}
\label{prop:holo nq}
Let $G$ be a finitely generated group. Then
\begin{equation*}
\label{eq:holo nilp}
\fh(G/\Gamma_k G)=\begin{cases}
\fh(G)/\fh(G)^{\prime} &\text{for $k=2$}, \\
\fh(G) &\text{for $k\ge3$}.
\end{cases}
\end{equation*}
In particular, the holonomy Lie algebra of $G$ depends 
only on the second nilpotent quotient, $G/\Gamma_3 G$. 
\end{prop}

\begin{proof}
The case $k=2$ is trivial, so let us assume $k\ge 3$. 
By a previous remark, the projection $G\to G/\Gamma_k G$ induces 
an isomorphism $\gr_2(G) \to \gr_2(G/\Gamma_k G)$. 
Furthermore, $H_1(G;\Q)\cong H_1(G/\Gamma_k G)$.
Using now the dual of the exact sequence \eqref{eq:sl}, we see that 
$\im(\mu^{\vee}_G)=\im(\mu^{\vee}_{G/\Gamma_k G})$. The desired conclusion follows.
\end{proof}

\begin{prop}
\label{prop:derivedquo}
The holonomy Lie algebras of the derived quotients of $G$ are given by
\begin{equation*}
\label{eq:holo derived}
\fh(G/G^{(i)})=\begin{cases}
\fh(G)/\fh(G)^{\prime} &\text{for $i=1$}, \\
\fh(G) &\text{for $i\ge 2$}.
\end{cases}
\end{equation*}
\end{prop}

\begin{proof}
For $i=1$, the statement trivially holds, so we may as well assume 
$i\ge 2$. 
It is readily proved by induction that $G^{(i)} \subseteq \Gamma_{2^i}(G)$.
Hence, the projections 
\begin{equation}
\label{eq:derivedseries}
\xymatrixcolsep{18pt}
\xymatrix{G\ar@{->>}[r] & G/G^{(i)}\ar@{->>}[r] &  G/\Gamma_{2^i}G}
\end{equation}
yield natural projections 
$\fh(G)\surj \fh(G/G^{(i)})\surj \fh(G/\Gamma_{2^i}G)=\fh(G)$.  
By Proposition \ref{prop:holo nq},
the composition of these projections 
is an isomorphism of Lie algebras. Therefore, 
the surjection $\fh(G)\surj \fh(G/G^{(i)})$ is an isomorphism.
\end{proof}

An analogous result holds for associated graded Lie algebras, 
albeit in somewhat weaker form. 
The quotient map, $q_i\colon G\surj G/G^{(i)}$, induces a surjective 
morphism between associated graded Lie algebras.  Plainly,  
this morphism is the canonical identification in degree $1$.  In fact, 
more is true. 

\begin{lemma}
\label{lem:griso}
For each $i\ge 2$ and each $k\leq 2^i-1$, the map 
$\gr(q_i)\colon \gr_k(G)\surj \gr_k(G/G^{(i)})$ is an isomorphism.
\end{lemma}

\begin{proof}
Taking associated graded Lie algebras in sequence 
\eqref{eq:derivedseries} yields epimorphisms 
\begin{equation}
\label{eq:derivedseries bis}
\xymatrixcolsep{18pt}
\xymatrix{\gr(G) \ar@{->>}[r] & \gr\big(G/G^{(i)}\big) \ar@{->>}[r] 
& \gr\big(G/\Gamma_{2^i}G\big)}.
\end{equation}
By a remark we made at the end of \S\ref{subsec:lcs}, the composition of these 
maps is an isomorphism in degrees $k<2^i$. The conclusion follows at once.
\end{proof}

The next result distills the statements of Theorem 9.3 and Corollary 9.5 
from \cite{SW-formality}, in a form needed here; this result sharpens and 
extends Theorem 4.2 from \cite{PS-imrn}.

\begin{theorem}[\cite{SW-formality}]
\label{thm:chenlie}
Let $G$ be a finitely generated group. 
For each $i\ge 2$, the quotient map $G\surj G/G^{(i)}$ induces a 
natural epimorphism of graded Lie algebras, 
$\gr(G)/\gr(G)^{(i)} \surj \gr(G/G^{(i)})$.  
Moreover, if $G$ is a $1$-formal group, then 
$\h(G)/\h(G)^{(i)} \cong \gr(G/G^{(i)})$.
\end{theorem}

Combining Theorem \ref{thm:chenlie} with Theorem \ref{thm:holo chen lie}, we obtain 
the following corollary.

\begin{corollary}
\label{cor:chenlie-formal}
Let $G=\langle \bx\mid \br\rangle$ be a finitely presented, $1$-formal group. 
Let $\by=\{y_1,\dots,y_b\}$ be a basis of $H_1(G;\Q)$.
Then, for each $i\ge 2$, 
\[
\gr(G/G^{(i)}) \cong \Lie(\by)/(  
{\rm ideal} (\kappa_2(w_{n-b+1}),\dots,\kappa_2(w_{m}))+ \Lie^{(i)}(\by)),
\]
where $b=b_1(G)$ and $w_k$ is defined in \eqref{eq:Gtilde}.
\end{corollary}

\subsection{Graded-formality}
\label{subsec:grf}

We conclude our discussion of associated graded 
Lie algebras and holonomy Lie algebras  
by recalling a notion that will be important in the sequel.  
Recall from Lemma \ref{lem:holoepi} that, for any finitely 
generated group $G$, there is a canonical epimorphism 
of graded Lie algebras, $\Phi_G\colon \fh(G)\surj \gr(G)$.  
We say that the group $G$ is {\em graded-formal}\/  (over $\Q$) 
if the map $\Phi_G$ is an isomorphism. 

This notion was considered in various ways by Chen \cite{Chen73}, 
Kohno \cite{Kohno}, Labute \cite{Labute85}, and Hain \cite{Hain85}. 
It was also recently studied by Lee in \cite{Lee},  
where it was called `graded $1$-formality.'   
Various relationships between graded-formality and other formality 
properties were studied in \cite{SW-formality}.  In particular, 
a finitely generated group $G$ is $1$-formal if and only if it is 
both graded-formal and filtered-formal. We give here two alternate 
characterizations of graded formality, which oftentimes are easier to verify.

\begin{prop}
\label{prop:graded-formal-cir}
A finitely generated group $G$ is graded-formal if and only if 
one of the following two conditions is satisfied.
\begin{enumerate}
\item \label{gfc1}  
The Lie algebra $\gr(G)$ is quadratic.
\item \label{gfc2}
$\dim_{\Q} \h_n(G) = \dim_{\Q} \gr_n(G)$, for all $n\ge 1$.
\end{enumerate}
\end{prop}

\begin{proof}
Clearly, if the group $G$ is graded-formal, then both conditions are satisfied. 

Assume now that \eqref{gfc1} holds, that is, the graded Lie algebra 
$\gr(G)$ admits a presentation of the form  
$\Lie(V)/\langle U \rangle$, where $V$ is a finite-dimensional 
$\Q$-vector space concentrated in degree $1$ and $U$ is a $\Q$-vector 
subspace of $\Lie_2(V)$.  In particular, $V=\gr_1(G)=H_1(G;\Q)$.  
From the exact sequence \eqref{eq:sl}, we see that the image of 
$\mu^{\vee}_G$ coincides with the kernel of the Lie bracket map 
$[\,,\,]\colon \bigwedge^2 \gr_1(G)\rightarrow \gr_2(G)$, which can 
be identified with $U$. Hence, the surjection $\varphi_G\colon \Lie(V)\surj \gr(G)$ 
induces an isomorphism $\Phi_G\colon \fh(G) \isom \gr(G)$.
 
Finally, assume \eqref{gfc2} holds. In general, 
the homomorphism $(\Phi_G)_n\colon \h_n(G)\to  \gr_n(G)$  
is an isomorphism for $n\le 2$ and an epimorphism for $n\ge 3$. 
Our assumption, together with the fact that each $\Q$-vector space 
$\fh_n(G)$ is finite-dimensional implies that all homomorphisms  
$(\Phi_G)_n$ are isomorphisms.  Therefore, the map 
$\Phi_G\colon \fh(G)\surj \gr(G)$ is an isomorphism of graded Lie algebras. 
\end{proof}

\section{Mildness and graded-formality}
\label{sect:mildness}

We start this section with the notion of mild (or inert) presentation 
of a group, due to J.~Labute and D.~Anick, and its relevance to the 
associated graded Lie algebra.  We then continue with some applications to 
two important classes of finitely presented groups:  
one-relator groups and fundamental groups of link complements. 

\subsection{Mild presentations}
\label{subsec:mild pres}
Let $F$ be a free group generated by
$\mathbf{x}=\{x_1,\dots,x_n\}$.  The \emph{weight}\/ of a word 
$r\in F$ is defined as $\omega(r)=\sup\{k\mid r\in \Gamma_k F\}$. 
Since $F$ is residually nilpotent, $\omega(r)$ is finite. 
The image of $r$ in $\gr_{\omega(r)}(F)$ is called the 
\textit{initial form}\/ of $r$, and is denoted by $\ini(r)$. 

Let $G=F/R$ be a quotient of $F$, with presentation 
$G=\langle \bx \mid \br \rangle$, where 
$\mathbf{r}=\{r_1,\dots,r_m\}$.  Let $\ini (\mathbf{r})$ be the 
ideal of the free $\Q$-Lie algebra $\Lie(\mathbf{x})$ 
generated by $\{\ini(r_1),\dots,\ini(r_m) \}$. Clearly, this 
is a homogeneous ideal; thus, the quotient  
\begin{equation}
\label{eq:liekg}
L(G):=\Lie(\mathbf{x})/\ini(\mathbf{r})
\end{equation}
is a graded Lie algebra. 
As noted by Labute in \cite{Labute85}, the ideal $\ini(\mathbf{r})$ 
is contained in $\gr^{\widetilde{\Gamma}}(R)$, where 
${\widetilde{\Gamma}}_k R= \Gamma_k{F}\cap R$ is the induced 
filtration on $R$. Hence, there exists an epimorphism $L(G)\surj \gr(G)$. 

\begin{prop}
\label{prop:labute comm}
Let $G$ be a commutator-relators group, and let $\fh(G)$ be its 
holonomy Lie algebra.  Then the canonical projection 
$\Phi_G\colon \fh(G)\surj \gr(G)$ factors through an epimorphism 
$\fh(G)\surj L(G)$.
\end{prop}

\begin{proof}
Let $G=\langle \bx \mid \br\rangle$ be a commutator-relators presentation 
for our group. By Corollary \ref{Cor:holo commrel}, the holonomy Lie 
algebra $\fh(G)$ admits a presentation of the form 
$\Lie(\bx)/\mathfrak{a}$, where $\mathfrak{a}$ is 
the ideal generated by the degree $2$ part of $M(r)-1$, for all $r\in 
\br$. On the other hand, $\ini(r)$ is the smallest degree homogeneous part 
of $M(r)-1$. Hence, $\mathfrak{a}\subseteq \ini(\br)$,  
and this complete the proof. 
\end{proof}

Following \cite{Anick87,Labute85}, we say that 
a group $G$ is a \textit{mildly presented group}\/ (over $\Q$) if it admits 
a presentation $G=\langle \mathbf{x}\mid \mathbf{r}\rangle$ such that the 
quotient $\ini(\mathbf{r})/[\ini(\mathbf{r}), \ini(\mathbf{r})]$, 
viewed as a $U(L(G))$-module via the adjoint representation of $L(G)$, 
is a free module on the images of $\ini(r_1),\dots,\ini(r_m)$.
As shown by Anick in \cite{Anick87}, a presentation 
$G=\langle x_1,\dots , x_n\mid  r_1, \dots r_m\rangle$ is mild if and only if 
\begin{equation}
\label{eq:anick criterion}
\Hilb(U(L(G)),t)= \left(1-nt+\sum_{i=1}^mt^{\omega(r_i)}\right)^{-1}.
\end{equation}

\begin{theorem}[Labute \cite{Labute70, Labute85}]
\label{thm: Labute85}
Let $G$ be a finitely-presented group. 
\begin{enumerate}
\item \label{mild1}
If $G$ is mildly presented, then $\gr(G)=L(G)$.  
\item  \label{mild2}
If $G$ has a single relator $r$, then $G$ is mildly presented.
Moreover, the LCS ranks $\phi_k(G):=\rank \gr_k(G)$ 
are given by
\begin{equation}
\label{eq:labutephi}
\phi_k(G)=\dfrac{1}{k}\sum_{d|k}\mu(k/d)\left[
\sum_{0\leq i\leq[d/e]} (-1)^i\dfrac{d}{d+i-ei}\, \binom{d+i-ie}{i}\, n^{d-ei}\right],
\end{equation}
where $\mu$ is the M\"{o}bius function and $e=\omega(r)$.
 \end{enumerate}
 \end{theorem}

Labute states this theorem over $\Z$, but his proof works for any commutative 
PID with unity.  There is an example in \cite{Labute85} showing that the mildness 
condition is crucial for part \eqref{mild1} of the theorem to hold. We give now a 
much simpler example to illustrate this phenomenon.

\begin{example}
Let $G=\langle x_1, x_2, x_3\mid  x_3, x_3[x_1,x_2]\rangle$. 
Clearly, $G\cong \langle x_1, x_2 \mid  [x_1,x_2]\rangle$, 
which is a mild presentation.  
However, the Lie algebra  $\Lie(x_1,x_2,x_3)/{\rm ideal}(x_3)$ 
is not isomorphic to
$\gr(G)=\Lie(x_1,x_2)/{\rm ideal}([x_1,x_2])$. 
Hence, the first presentation is not a mild.
\end{example}

Under different assumptions, alternative methods for computing the LCS ranks of a group $G$
can be found in \cite{Weigel15, SW-mz}.

\subsection{Mildness and graded formality}
\label{subsec:apps graded}
We now use Labute's work on the associated graded Lie algebra 
and our presentation of the holonomy Lie algebra  
to give two graded-formality criteria.  

\begin{corollary}
\label{cor:graded-1-formal}
Let $G$ be a group admitting a mild presentation 
$\langle \mathbf{x}\mid  \mathbf{r}\rangle$.  If $\omega(r)\le 2$ 
for each $r\in \mathbf{r}$, then $G$ is graded-formal.
\end{corollary}

\begin{proof}
By Theorem \ref{thm: Labute85}, the associated graded Lie algebra 
$\gr(H;\Q)$ has a presentation of the form $\Lie(\mathbf{x})/\ini(\mathbf{r})$, 
with $\ini(\mathbf{r})$ a homogeneous ideal generated in degrees $1$ and $2$. 
Using the degree $1$ relations to eliminate superfluous generators, 
we arrive at a presentation with only quadratic relations. 
The desired conclusion follows from Proposition \ref{prop:graded-formal-cir}.
\end{proof}

An important sufficient condition for mildness of a presentation 
was given by Anick \cite{Anick87}.
Recall that $\iota$ denotes the canonical injection 
from the free Lie algebra $\Lie(\bx)$ into $\Q \langle \bx\rangle$. 
Fix an ordering on the set $\{\bx\}$. The set of 
monomials in the homogeneous elements $\iota(\ini(r_1)),\dots,\iota(\ini(r_m))$ 
inherits the lexicographic order.
Let $w_i$ be the highest term of $\iota(\ini(r_i))$ for $1\leq i\leq m$. 
Suppose that
(i) no $w_i$ equals zero;
 (ii) no $w_i$ is a submonomial of any $w_j$ for $i\neq j$, i.e., 
$w_j=uw_iv$ cannot occur;
and (iii) no $w_i$ overlaps with any $w_j$, i.e., $w_i=uv$ and 
$w_j=vw$ cannot occur unless $v=1$, or $u=w=1$. 
Then, the set $\{r_1,\dots,r_n\}$ is mild (over $\Q$). 
We use this criterion to provide an  example of a finitely-presented 
group $G$ which is graded-formal, but not filtered-formal.
 
\begin{example}
\label{ex:quadrelation}
Let $G$ be the group with generators 
$x_1, \dots, x_4$  and relators $r_1=[x_2,x_3]$, 
$r_2=[x_1,x_4]$, and $r_3=[x_1,x_3][x_2,x_4]$.  
Ordering the generators as $x_1\succ x_2\succ x_3\succ x_4$,  
we find that the highest terms for $\{\iota(\ini(r_1)), \iota(\ini(r_2)), \iota(\ini(r_3))\}$ 
are $\{x_2x_3, x_1x_4, x_1x_3\}$, and these words satisfy the above 
conditions of Anick.  Thus, by Theorem \ref{thm: Labute85}, the Lie algebra 
$\gr(G)$ is the quotient of $\Lie(x_1,\dots,x_4)$ by the 
ideal generated by $[x_2,x_3]$, $[x_1,x_4]$, and $[x_1,x_3]+[x_2,x_4]$.
Hence, $\fh(G)\cong\gr(G)$, that is, $G$ is graded-formal.
On the other hand, using the Tangent Cone Theorem from 
\cite{DPS}, one can show that the group $G$ is not $1$-formal.  
Therefore, $G$ is not filtered-formal.
\end{example}

\subsection{The rational Murasugi conjecture}
\label{subsec:murasugi}
Let $L=(L_1,\dots , L_n)$ be an $n$-component link in $S^3$.  
The complement of the link, $X=S^3\setminus \bigcup_{i=1}^n L_i$,  
is a a connected, $3$-dimensional manifold, which has the homotopy 
type of a finite, $2$-dimensional CW-complex.   
The link group, $G=\pi_1(X)$, carries crucial information about 
the homotopy type of $X$.  For instance, if $n=1$ (i.e., the link is a knot), 
or if $n>1$ and $L$ is a not a split link, then the complement $X$ is a 
$K(G,1)$.

Let $\ell_{i,j}=\lk(L_i,L_j)$ be the linking numbers of $L$.  
The information coming from these numbers is conveniently encoded 
in a graph $\Gamma$ with vertex set $\{1,\dots ,n\}$, 
and edges $(i,j)$ whenever $\ell_{i,j} \ne 0$.
As noted in \cite{Markl-Papadima, PS-imrn}, the holonomy Lie algebra 
$\h(G)=\h(X)$ is determined by these data: 
\begin{equation}
\label{eq:hlielink}
\h(G)= \Lie(y_1,\dots ,y_n)\Big/\bigg(\sum_{j=1}^n \ell_{i,j}[y_i,y_j]=0,
\ 1\le i<n\bigg)\, .
\end{equation}
 
Turning now to the associated graded Lie algebra of a link group $G$, 
Murasugi conjectured in \cite{Murasugi} that $\gr_k(G;\Z)=\gr_k(F_{n-1};\Z)$ 
for all $k>1$, provided that the link $L$ has $n$ components, and all the linking
numbers are equal to $\pm 1$. This conjecture was proved by Massey--Traldi 
\cite{MasseyTraldi} and Labute \cite{Labute90}, who also proved an 
analogous result for the Chen ranks of such links.  In \cite{Traldi89}, 
Traldi computed the Chen groups $\gr_k(G/G'';\Z)$ for all links with 
connected linking graph.  The next theorem, which can be 
viewed as a rational version of Murasugi's conjecture, combines results 
of Anick \cite{Anick87}, Berceanu--Papadima \cite{Berceanu-Papadima94}, 
and Papadima--Suciu \cite{PS-imrn}.
 
\begin{theorem}
\label{thm:linkgroups}
Let $L$ be an $n$-component link in $S^3$, and let $G$ be the fundamental 
group of its complement.  Assuming the linking 
graph $\Gamma$ is connected, the following hold.
\begin{enumerate}
\item \label{g1} The group $G$ is graded-formal, and thus, 
the associated graded Lie algebra
$\gr(G)$ is isomorphic to the holonomy Lie algebra $\fh(G)$, 
with presentation given by \eqref{eq:hlielink}.
\item \label{g2} There exists a graded Lie algebra isomorphism 
$\gr(G/G'')\cong \fh(G)/\fh(G)''$.
\item \label{g3} 
If, furthermore, $L$ is the closure of a pure braid, then $G$ 
admits a mild presentation. 
\end{enumerate}
\end{theorem}

\begin{proof}
The first assertion follows from Lemma 4.1 and Theorems 3.2 and 4.2 in 
\cite{Berceanu-Papadima94}, the second assertion  is proved in 
\cite[Thm.~10.1]{PS-imrn}, while the last assertion follows 
from  \cite[Thm.~ 3.7]{Anick87}.
\end{proof}
 
We conclude this section with several examples illustrating the concepts 
discussed above.  In each example, $L$ is a link in $S^3$, and $G$ is the 
corresponding link group. The first two examples were computed by
Hain in \cite{Hain85} using a slightly different method. 

 \begin{example}
\label{ex:borromean}
Let $L$ be the Borromean rings. 
All the linking numbers are $0$, and so $\h(G)=\Lie(x,y,z)$.  
The link group $G$ has a presentation with three generators $x,y,z$ 
and two relators, $r_1=[x,[y,z]]$ and $r_2=[z,[y,x]]$.
It is readily seen that the link $L$ passes Anick's 
mildness test; hence $G$  admits a mild presentation.  
Thus, $\gr(G)=\Lie(x,y,z)/\ideal ([x,[y,z]], [z,[y,x]])$, and   
so $G$ is not graded-formal.
\end{example}

\begin{example}
\label{ex:whitehead}
Let $L$ be the Whitehead link.  This is a $2$-component link with 
linking number $0$; its link group has presentation   
$G = \langle  x, y\mid r \rangle$, where 
$r=x^{-1}y^{-1}xyx^{-1}yxy^{-1}xyx^{-1}y^{-1}xy^{-1}x^{-1}y$.  
Since $G$ has only one relator, Theorem \ref{thm: Labute85} 
insures that this presentation is mild.
Direct computation shows that $\ini(r)=[x,[y,[x,y]]]$.   
Thus, $\gr(G)=\Lie(x,y)/\ideal ([x,[y,[x,y]]] )$,  
and $G$ is not graded-formal. 
\end{example}

\begin{example}
\label{ex:semiproduct}
Let $L$ be the link of great circles in $S^3$ corresponding to the 
arrangement of transverse planes through the origin of $\R^4$ 
denoted as $\mathcal{A}(31425)$ in Matei--Suciu \cite{MS-top}.  
By construction, $L$ is the closure of a pure braid, and its linking 
graph is a complete graph. Thus, by 
Theorem \ref{thm:linkgroups}, the link group 
$G$ is graded-formal, and admits a mild presentation.   
On the other hand, as noted in 
\cite[Example 8.2]{DPS}, the group $G$ is not $1$-formal.
\end{example}

\section{One-relator groups}
\label{sec:onerel}

We turn now to some other specific classes of finitely presented groups 
where our approach applies.  We start with a well-known and much-studied 
class of groups in group theory.

\subsection{Holonomy and graded-formality}
\label{subs:hologr}
If the group $G$ admits a finite presentation with a single relator, 
we saw in the previous section that $G$ is mildly presented.  
In fact, more can be said. 

\begin{prop}
\label{prop:1rel}
Let $G=\langle \bx\mid r\rangle$ be a $1$-relator group. 
\begin{enumerate}
\item \label{e1}
If $r$ is a commutator relator, then $\fh(G)=\Lie(\bx)/{\rm ideal}(M_2(r))$. 
\item \label{e2}
If $r$ is not a commutator relator, then  $\fh(G)=\Lie(y_1,\dots,y_{n-1})$.
\end{enumerate}
\end{prop}

\begin{proof}
Part \eqref{e1} follows from Corollary \ref{Cor:holo commrel}.  
When $r$ is not a commutator relator, the Jacobian matrix 
$J_G=(\varepsilon(\partial_i r))$ has rank $1$.
Part  \eqref{e2} then follows from Theorem \ref{ThmholonomyLie}.
\end{proof}

\begin{corollary}
\label{cor:Hilb1rel}
Let $G=\langle x_1,\dots x_n \mid r\rangle$ be a $1$-relator group, 
and let $\h=\h(G)$.  Then  
\begin{equation}
\Hilb(U(\fh);t)=
\begin{cases}
1/(1-(n-1)t)&\text{if $\omega(r)=1$},\\
1/(1-nt+t^2)&\text{if $\omega(r)=2$},\\
1/(1-nt)&\text{if $\omega(r)\geq 3$}.
\end{cases}
\end{equation}
\end{corollary}

\begin{proof}
Let $\bx=\{x_1,\dots,x_n\}$. 
By Proposition \ref{prop:1rel}, the universal enveloping algebra $U(\fh)$ 
is isomorphic to either $T(y_1,\dots,y_{n-1})$ if $\omega(r)=1$, or to 
$T(\bx)/\ideal(M_2(r))$ if $\omega(r)=2$, or to  $T(\bx)$ if  
$\omega(r)\geq 3$. The claim now follows from 
Theorem \ref{thm: Labute85} and formula \eqref{eq:anick criterion}.
\end{proof}

\begin{theorem}
\label{thm:1rel}
Let $G=\langle \mathbf{x}\mid r\rangle$ be a group defined 
by a single relation. Then $G$ is graded-formal if and only if 
$\omega(r)\le 2$.
\end{theorem}

\begin{proof}
By Theorem \ref{thm: Labute85}, the given presentation of $G$ is mild. 
The weight $\omega(r)$ can also be computed as  
$\omega(r)=\inf\{\abs{I} \mid M(r)_I\neq 0\}$.  
If $\omega(r)\leq 2$, then, by Proposition \ref{prop:1rel}, 
we have that 
$\fh(G)\cong \gr(G)\cong \Lie(\mathbf{x})/{\rm ideal}(\ini(r))$, 
and so $G$ is graded-formal. 

On the other hand, if $\omega(r)\geq 3$, then $\fh(G)=\Lie(\mathbf{x})$.  
However, $\gr(G)=\Lie(\mathbf{x})/{\rm ideal}(\ini(r))$. 
Hence, $G$ is not graded-formal. 
\end{proof}

\begin{example}
\label{ex:omega3}
Let $G=\langle x_1, x_2\mid r\rangle$, where $r=[x_1,[x_1,x_2]]$. 
Clearly,  $\omega(r)=3$. Hence, $G$ is not graded-formal.
\end{example}

The next example shows that there exists a graded-formal group 
which is not filtered-formal.

\begin{example}
\label{ex:1relnot1formal}
Let $G=\langle x_1,\dots,x_5\mid w\rangle$, where $w=[x_1,x_2][x_3,[x_4,x_5]]$. 
Since $\omega(w)=2$, Theorem \ref{thm:1rel} implies that the group $G$ is graded-formal.
On the other hand,  as shown in \cite{SW-formality}, $G$ is not $1$-formal, and so 
$G$ is not filtered-formal. 
\end{example}

\subsection{Chen ranks}
\label{subsec:chen ranks}

We now determine the ranks of the (rational) Chen Lie algebra  associated 
to an arbitrary finitely presented, $1$-relator, $1$-formal group, thereby extending  
a result of Papadima and Suciu from \cite{PS-imrn}. 
By definition, the {\em Chen ranks}\/ of a finitely generated group $G$ 
are the LCS ranks of its maximal metabelian quotient, 
\begin{equation}
\label{eq:theta}
\theta_k(G):=\dim_{\Q} (\gr_k(G/G^{\prime\prime})).
\end{equation}

The projection $\pi\colon G\surj G/G^{\prime\prime}$ induces 
an epimorphism, $\gr(\pi) \colon \gr(G)\surj \gr(G/G^{\prime\prime})$, 
which is an isomorphism in degrees $k\le 3$.  Consequently, $\phi_k(G)\ge \theta_k(G)$, 
with equality for $k\le 3$. The Chen ranks were introduced and 
studied by K.-T. Chen \cite{Chen51}, who showed that, 
for all $k\ge 2$, 
\begin{equation}
\label{eq:chenrankF}
\theta_k(F_n)=(k-1)\binom{n+k-2}{k}. 
\end{equation}

The {\em holonomy Chen ranks}\/ of the group $G$ are defined as 
$\bar\theta_k(G):= \dim (\fh/\fh'')_k$, where $\h=\h(G)$. 
It is readily seen that $\bar\theta_k(G)\ge \theta_k(G)$, 
with equality for $k\le 2$.   A basic 
result in the subject reads as follows:  If $G$ is $1$-formal, then 
\begin{equation}
\label{eq:chen equal}
\theta_k(G)= \bar\theta_k(G),
\end{equation}
for all $k\geq 1$.  This result was proved in \cite[Cor.~9.4]{PS-imrn} 
for $1$-formal groups admitting a finite, 
commutator-relators presentation, and in full generality in 
 \cite[Prop.~8.1 and Cor.~8.6]{SW-mz}.
 
\begin{prop}
\label{prop:1rel-chen}
Let $G=F/\langle r\rangle$ be a one-relator group, where 
$F=\langle x_1,\dots , x_n\rangle$, 
and suppose $G$ is $1$-formal. Then
\[
\Hilb(\gr(G/G^{\prime\prime}),t)=
\begin{cases}
1+nt-\dfrac{1-nt+t^2}{(1-t)^n} & \text{if $r\in [F,F]$},
\\[8pt]
1+(n-1)t-\dfrac{1-(n-1)t}{(1-t)^{n-1}} & \text{otherwise}.
\end{cases}
\]
\end{prop}

\begin{proof}
First assume that $r\in [F,F]$. The claim 
is then proved in \cite[Thm.~7.3]{PS-imrn}. 

Now assume that $r\notin [F,F]$. In that case, Theorem \ref{ThmholonomyLie} 
implies that $\fh(G)\cong \Lie(y_1,\dots, y_{n-1})$, which in turn 
is isomorphic to $\fh(F_{n-1})$. 
Since both $G$ are $F_{n-1}$ is $1$-formal, formula \eqref{eq:chen equal} 
implies that 
\begin{equation}
\label{eq:chen stuff}
\theta_k(G)=\bar\theta_k(G)=\bar\theta_k(F_{n-1})=\theta_k(F_{n-1}).
\end{equation}
The claim then follows from Chen's formula \eqref{eq:chenrankF}.
\end{proof}

\subsection{Surface groups}
\label{subsec:surf gp}
The Riemann surface $\Sigma_g$ is a compact K\"{a}hler manifold, 
and thus, a formal space.   The formality of $\Sigma_g$ 
also implies the $1$-formality of $\Pi_g$.  As a consequence, the associated 
graded Lie algebra $\gr(\Pi_g)$ is isomorphic to the holonomy Lie 
algebra $\fh(\Pi_g)$, which has a presentation
$\fh(\Pi_g) = \Lie(2g )\Big/\Big\langle \sum_{i=1}^{g} [x_i,y_i]=0 \Big\rangle$,
where $\Lie(2g):=\Lie(x_1, y_1,\dots, x_g, y_g)$.  
It follows from formula \eqref{eq:labutephi} that 
the LCS ranks of the $1$-relator group 
$\Pi_g$ are given by 
\begin{equation}
\label{eq:lcsRiemann}
\phi_k(\Pi_g)=\dfrac{1}{k}\sum_{d|k}\mu(k/d)\left[\sum_{i=0}^{[d/2]} 
(-1)^i\dfrac{d}{d-i}\binom{d-i}{i} (2g)^{d-2i}\right].
\end{equation}
Using now Theorem \ref{thm:holo chen lie}, we see that the 
Chen Lie algebra of $\Pi_g$ has presentation 
\begin{equation}
\label{eq:chenRiemann}
\gr \big(\Pi_g/\Pi_g''\big)=\Lie(2g)\Big/ 
\Big\langle 
\Big\langle \sum_{i=1}^{g} [x_i,y_i] \Big\rangle+ \Lie''(2g) 
\Big\rangle. 
\end{equation}
Furthermore, Proposition \ref{prop:1rel-chen} shows that the Chen 
ranks of our surface group are given by $\theta_1(\Pi_g)=2g$, 
$\theta_2(\Pi_g)=2g^2-g-1$, and 
\begin{equation}
\label{eq:chenRiemann ranks}
\theta_k(\Pi_g)=(k-1)\binom{2g+k-2}{k}-\binom{2g+k-3}{k-2}, \: \textrm{ for } k\geq 3.
\end{equation} 
 
Let $N_h$ be the nonorientable surface of genus $h\geq 1$. 
It is well known that $N_h$ has the rational homotopy type of
a wedge of $h-1$ circles, see \cite[Example 6.18]{DPS}.
Hence, $N_h$ is a formal space, and thus $\pi_1(N_h)$ is a $1$-formal group.
Furthermore,  Proposition \ref{prop:1rel}
shows that the holonomy Lie algebra of $\pi_1(N_h)$ is isomorphic to 
the free Lie algebra with $h-1$ generators, and 
Proposition \ref{prop:1rel-chen} implies that the Chen ranks 
of $\pi_1(N_h)$ are given by $\theta_k(\pi_1(N_h))=(k-1)\binom{h+k-3}{k}$ 
for $k\geq 2$.

\section{Seifert fibered spaces}
\label{sec:seifert mfd}

We will consider here only orientable, closed Seifert manifolds 
with orientable base.   Every such manifold $M$ admits an effective circle action, 
with orbit space an orientable surface of genus $g$, and finitely many 
exceptional orbits, encoded in pairs of coprime integers $(\alpha_1,\beta_1), 
\dots,  (\alpha_s,\beta_s)$ with $\alpha_j\ge 2$. The obstruction to trivializing 
the bundle $\eta\colon M\to \Sigma_g$ outside tubular neighborhoods of the exceptional 
orbits is given by an integer $b=b(\eta)$.  A standard presentation for the 
fundamental group of $M$ in terms of the Seifert invariants is given by
\begin{equation}
 \begin{split}
 \pi_{\eta}:=\pi_1(M)&=\big\langle x_1, y_1,\dots, x_g, y_g, z_1,\dots,z_s, h\mid  
 h ~\textrm{central}, \\
&[x_1,y_1]\cdots[x_g,y_g]z_1\cdots z_s=h^{b},\: 
 z_i^{\alpha_i}h^{\beta_i}=1\: (i=1,\dots,s) \big\rangle.
 \end{split}
 \end{equation}
As shown by Scott in \cite{Scott83}, the Euler number $e(\eta)$ of  the 
Seifert bundle $\eta \colon M\rightarrow \Sigma_g$ satisfies
$e(\eta)=-b(\eta)-\sum_{i=1}^{s}\beta_i/\alpha_i$.

\subsection{Holonomy Lie algebra} 
\label{subsec:holo seifert}
 
We now give a presentation for the holonomy Lie algebra of a Seifert 
manifold group. 

\begin{theorem}
\label{thm:seifertholo}
Let $\eta\colon M\to \Sigma_g$ be a Seifert fibration with orientable base.   
The rational holonomy Lie algebra of the group $\pi_{\eta}=\pi_1(M)$ 
is given by
\[
\fh(\pi_{\eta};\Q)=
\begin{cases}
\Lie(x_1,y_1,\dots,x_g,y_g,h )/\langle 
\sum_{i=1}^{s} [x_i,y_i]=0, h {~\rm central} \rangle 
 & \text{if $e(\eta)=0$};
 \\
\Lie( 2g)
 & \text{if $e(\eta)\ne 0$}.
\end{cases}
\]
 \end{theorem}
 
\begin{proof}
First assume $e(\eta)=0$. In this case, the row-echelon approximation 
of $\pi_{\eta}$ has presentation 
\begin{equation}
\label{eq:pitilde}
 \begin{split}
\widetilde{\pi}_{\eta}&=\langle x_1, y_1,\dots, x_g, y_g, z_1,\dots,z_s, h\mid  
 z_i^{\alpha_i}h^{\beta_i}=1\: (i=1,\dots,s),
\\
&\qquad \qquad ([x_1,y_1]\cdots[x_g,y_g])^{\alpha_1\cdots\alpha_s}=1,\:  
h ~\textrm{central} 
 \rangle
 \end{split}
 \end{equation}
It is readily seen that the rank of the Jacobian matrix associated 
to this presentation has rank $s$.  
Furthermore, the map $\pi \colon F_{\Q}\rightarrow H_{\Q}$ is given 
by $x_i\mapsto x_i$, $y_i\mapsto y_i$, $z_j\mapsto (-\beta_i/\alpha_i)h$, $h\mapsto h$.  
Let $\kappa$ be  the Magnus expansion from Definition \ref{def:kappa}.  
A Fox Calculus computation shows that $\kappa$ takes the 
following values on the commutator-relators of $\widetilde{\pi}_{\eta}$: $\kappa([z_i,h])=1,$
\[
\begin{aligned}
\kappa([x_i,h])&=1+x_ih-hx_i+\textrm{terms ~of ~  degree}\geq 3,\\
\kappa([y_i,h])&=1+y_ih-hy_i+\textrm{terms ~of ~  degree}\geq 3, \\
\kappa (r) &=
1+(\alpha_1\cdots\alpha_s)(x_1y_1-y_1x_1+\dots+x_gy_g-y_gx_g)
+\textrm{terms ~of ~  degree}\geq 3, 
\end{aligned}
\]
where $r=([x_1,y_1]\cdots[x_g,y_g])^{\alpha_1\cdots\alpha_s}$.
The first claim now follows from Theorem \ref{ThmholonomyLie}.
 
Next, assume $e(\eta)\neq0$.  
Then the row-echelon approximation of $\pi_\eta$ is given by
\begin{equation}
\label{eq:pitilde-bis}
 \begin{split}
\widetilde{\pi}_{\eta}&=\langle x_1, y_1,\dots, x_g, y_g, z_1,\dots,z_s, h\mid  
 z_i^{\alpha_i}h^{\beta_i}=1\: (i=1,\dots,s),
\\
&\qquad \qquad ([x_1,y_1]\cdots
[x_g,y_g])^{\alpha_1\cdots\alpha_s}h^{e(\eta)\alpha_1\cdots\alpha_s}=1,\:  
h ~\textrm{central} 
\rangle,
\end{split}
\end{equation}
while the homomorphism $\pi\colon  F_{\Q}\rightarrow H_{\Q}$ is 
given by $x_i\mapsto x_i$, $y_i\mapsto y_i$, $z_j\mapsto (-\beta_i/\alpha_i)h$, 
and $h\mapsto 0$.  As before, the second claim follows from 
Theorem \ref{ThmholonomyLie}, and we are done.
\end{proof}

The Malcev Lie algebra of $\pi_{\eta}$, given in \cite[Thm.11.6]{SW-formality},
has an explicit presentation, which is the degree completion of the 
graded Lie algebra
\begin{equation}
\label{eq:seifertmalcev}
L(\pi_{\eta})=
\begin{cases}
\Lie(x_1,y_1,\dots,x_g,y_g,z )/\langle 
\sum_{i=1}^{g} [x_i,y_i]=0,\: z {~\rm central} \rangle 
 & \text{if $e(\eta)=0$};
 \\[2pt]
\Lie( x_1,y_1,\dots,x_g,y_g,w)/\langle 
\sum_{i=1}^{g} [x_i,y_i]=w,\: w {~\rm central} \rangle 
 & \text{if $e(\eta)\ne 0$},
\end{cases}
\end{equation}
where $\deg(w)=2$ and the other generators have degree $1$.
Moreover, $\gr(\pi_{\eta})\cong L(\pi_{\eta})$. 
From the presentation of $\pi_{\eta}$ and the definition of filtered formality, 
 we immediately obtain that fundamental groups of 
 orientable Seifert manifolds are filtered-formal.

\subsection{LCS ranks} 
\label{subsec:lcs seifert}
We end this section with a computation of the ranks 
of the various graded Lie algebras attached to the fundamental 
group of an orientable Seifert manifold. Comparing these ranks, we derive 
some consequences regarding the non-formality properties 
of such groups.

We start with the LCS ranks $\phi_k(\pi_{\eta})=\dim \gr_k(\pi_{\eta})$ 
and the holonomy ranks $\bar\phi_k(\pi_{\eta})=\dim \fh(\pi_{\eta})_k$. 

\begin{prop}
\label{prop:seifertlcs} 
The LCS ranks and the holonomy ranks of a Seifert manifold 
group $\pi_{\eta}$ are computed as follows.
\begin{enumerate}
\item \label{s1}
If $e(\eta)=0$, then $\phi_1(\pi_{\eta})=\bar\phi_1(\pi_{\eta})=2g+1$, and 
$\phi_k(\pi_{\eta})=\bar\phi_k(\pi_{\eta})=\phi_k(\Pi_{g})$ for $k\geq 2$.

\item  \label{s2}
If $e(\eta)\neq 0$, then 
$\bar\phi_k(\pi_{\eta})=\phi_k(F_{2g})$ for $k\geq 1$.

\item  \label{s3}
If $e(\eta)\neq 0$, then $\phi_1(\pi_{\eta})=2g$,   
$\phi_2(\pi_{\eta})=g(2g-1)$, 
and $\phi_k(\pi_{\eta})=\phi_k(\Pi_{g})$ for $k\geq 3$.
\end{enumerate}
Here the LCS ranks $\phi_k(\Pi_{g})$ are given by 
formula \eqref{eq:lcsRiemann}. 
\end{prop}
  
\begin{proof}
If $e(\eta)=0$, then $\pi_{\eta}\cong \Pi_{g} \times \Z$, and 
claim \eqref{s1} readily follows. So suppose that $e(\eta)\neq 0$.  
In this case, we know from Theorem \ref{thm:seifertholo} 
that $\fh(\pi_{\eta})=\fh(F_{2g})$, and thus claim \eqref{s2} follows. 

By \eqref{eq:seifertmalcev}, the associated graded Lie algebra   
$\gr(\pi_{\eta})$ is isomorphic to the quotient 
of the free Lie algebra $\Lie( x_1,y_1,\dots,x_g,y_g,w)$  
by the ideal generated by the elements $\sum_{i=1}^{g} [x_i,y_i]-w$,  $[w,x_i]$, 
and $[w,y_i]$.   Define a morphism $\chi \colon \gr(\pi_{\eta})\to  
\gr(\Pi_g)$ by sending  $x_i\mapsto x_i$, $y_i\mapsto y_i$, and $w\mapsto 0$.
It is readily seen that the kernel of $\chi$ is the Lie ideal of 
$\gr(\pi_{\eta})$ generated by $w$, and this ideal is isomorphic 
to the free Lie algebra on $w$.   Thus, we get a short exact sequence 
of graded Lie algebras,
\begin{equation}
\label{eq:centext}
\xymatrix{0\ar[r]& \Lie(w) \ar[r]& \gr(\pi_{\eta}) \ar^{\chi}[r]& 
 \gr(\Pi_g) \ar[r]& 0
}.
\end{equation}
Comparing Hilbert series in this sequence establishes claim \eqref{s3} 
and completes the proof. 
\end{proof}

\begin{corollary}
\label{cor:seifertgraded}
If $g=0$, the group $\pi_{\eta}$ is always $1$-formal,  
while if $g>0$, the group $\pi_{\eta}$ is graded-formal if and only if $e(\eta)= 0$. 
\end{corollary} 

\begin{proof}
First suppose $e(\eta)=0$.  In this case, we know from \eqref{eq:seifertmalcev} 
that $\gr(\pi_{\eta})\cong\gr(\Pi_g)\times \gr (\Z)$.  
It easily follows that $\gr(\pi_{\eta})\cong \fh(\pi_{\eta})$ 
by comparing the presentations of these two Lie algebras. 
Hence, $\pi_{\eta}$ is graded-formal, and thus $1$-formal, by 
the fact that $\pi_{\eta}$ is filtered formal. 

It is enough to assume that $g>0$ and $e(\eta)\neq 0$, since 
the other claims are clear.  By Proposition \ref{prop:seifertlcs}, we 
have that  $\bar{\phi}_3(\pi_{\eta})=(8g^3-2g)/3$,  
whereas $\phi_3(\pi_{\eta})= (8g^3-8g)/3$. Hence, $\h(\pi_{\eta})$ is 
not isomorphic to $\gr(\pi_{\eta})$, proving that $\pi_{\eta}$ is not 
graded-formal.
\end{proof}

\subsection{Chen ranks} 
\label{subsec:chen seifert}
Recall that the Chen ranks are defined as $\theta_k(\pi)=\dim \gr_k(\pi/\pi'')$, 
while the holonomy Chen ranks are defined as $\bar\theta_k(\pi)=\dim (\h/\h'')_k$, 
where $\fh=\fh(\pi)$. 
 
\begin{prop}
\label{prop:seifert-chen}
The Chen ranks and the holonomy Chen ranks of a Seifert manifold 
group $\pi_{\eta}$ are computed as follows.
\begin{enumerate}
\item \label{sf1}
If $e(\eta)=0$, then $\theta_1(\pi_{\eta})=\bar\theta_1(\pi_{\eta})=2g+1$, and 
$\theta_k(\pi_{\eta})=\bar\theta_k(\pi_{\eta})=\theta_k(\Pi_{g})$ for $k\geq 2$.

\item  \label{sf2}
If $e(\eta)\neq 0$, then 
$\bar\theta_k(\pi_{\eta})=\theta_k(F_{2g})$ for $k\geq 1$.

\item  \label{sf3}
If $e(\eta)\neq 0$, then $\theta_1(\pi_{\eta})=2g$,   
$\theta_2(\pi_{\eta})=g(2g-1)$, and
$\theta_k(\pi_{\eta})=\theta_k(\Pi_{g})$ for $k\geq 3$.
\end{enumerate}
Here the Chen ranks $\theta_k(F_{2g})$ and $\theta_k(\Pi_{g})$ are given by 
formulas \eqref{eq:chenrankF} and \eqref{eq:chenRiemann ranks}, respectively. 
\end{prop}

\begin{proof}
Claims \eqref{sf1} and \eqref{sf2} are easily proved, as in 
Proposition \ref{prop:seifertlcs}.  To prove claim \eqref{sf3}, 
start by recalling that 
the group $\pi_{\eta}$ is filtered-formal. 
Hence, from \cite[Theorem 9.3]{SW-formality}, 
the Chen Lie algebra $\gr(\pi_{\eta}/\pi_{\eta}'')$ is 
isomorphic to $\gr(\pi_{\eta})/\gr(\pi_{\eta})''$.  
As before, we obtain a short exact sequence of graded Lie algebras,
\begin{equation}
\label{eq:centextension}
\xymatrixcolsep{18pt}
\xymatrix{0\ar[r]& \Lie(w) \ar[r]& \gr(\pi_{\eta}/\pi_{\eta}'') \ar[r]& 
 \gr(\Pi_g/\Pi_g'') \ar[r]& 0
}.
\end{equation}
Comparing Hilbert series in this sequence completes the proof. 
\end{proof}

\begin{remark}
\label{rem:theta3}
The above result can be used to give another proof of Corollary \ref{cor:seifertgraded}. 
Indeed, suppose $e(\eta)\neq 0$.  Then, by Proposition \ref{prop:seifert-chen}, 
we have that $\bar{\theta}_3(\pi_{\eta})-\theta_3(\pi_{\eta})=2g$. Consequently, 
the group $\pi_{\eta}$ is not $1$-formal. 
The group $\pi_{\eta}$ is not graded-formal, since it is filtered-formal. 
\end{remark}

\newcommand{\arxiv}[1]
{\texttt{\href{http://arxiv.org/abs/#1}{arXiv:#1}}}
\newcommand{\arx}[1]
{\texttt{\href{http://arxiv.org/abs/#1}{arxiv:}}
\texttt{\href{http://arxiv.org/abs/#1}{#1}}}
\newcommand{\doi}[1]
{\texttt{\href{http://dx.doi.org/#1}{doi:#1}}}
\renewcommand*\MR[1]{%
\StrBefore{#1 }{ }[\temp]%
\href{http://www.ams.org/mathscinet-getitem?mr=\temp}{MR#1}}

\bibliographystyle{amsplain}

\end{document}